\documentclass[11pt,reqno]{article}
\usepackage{amsmath,amsfonts,amsthm}
\usepackage[margin=1in]{geometry}
\usepackage[normalem]{ulem}

\usepackage{ amssymb }
\usepackage[hyphens]{url}
\usepackage{hyperref}
\usepackage{graphicx}
\usepackage{verbatim}
\usepackage{enumitem} 
\usepackage{color}

\numberwithin{equation}{section}

\newtheorem{theorem}{Theorem}[section]
\newtheorem{assumption}{Assumption}[section]
\newtheorem{corollary}{Corollary}[section]
\newtheorem{lemma}{Lemma}[section]
\newtheorem{proposition}{Proposition}[section]

\newtheorem{definition}{Definition}[section]
\newtheorem{remark}{Remark}[section]
\newtheorem{example}{Example}[section]

\newcommand{\lemref}{Lemma~\ref}

\newcommand{\propref}{Proposition~\ref}

\newcommand{\asref}{Assumption~\ref}

\renewcommand{\P}{\mathbb{P}}
\newcommand{\Q}{\mathbb{Q}}
\newcommand{\R}{\mathbb{R}}
\newcommand{\E}{\mathbb{E}}
\newcommand{\cE}{\mathcal{E}}

\newcommand{\N}{\mathbb{N}}

\newcommand{\F}{\mathcal{F}}

\newcommand{\X}{\mathbb{X}}

\newcommand{\B}{\mathcal{B}}

\newcommand{\T}{\mathcal{T}}

\newcommand{\eps}{\varepsilon}

\newcommand{\nada}[1]{}
\definecolor{gb}{rgb}{0, 0.2, 0.8}

\title{Optimal Equilibria for Time-Inconsistent Stopping Problems in Continuous Time\thanks{We would like to thank the anonymous referee and Associate Editor, whose suggestions led to stronger results and nicer presentation.}}
\author{Yu-Jui Huang\thanks{
University of Colorado, Department of Applied Mathematics, Boulder, CO 80309-0526, USA, email: \texttt{yujui.huang@colorado.edu}. Partially supported by National Science Foundation (DMS-1715439) and the University of Colorado (11003573).}
 \and Zhou Zhou\thanks{
University of Sydney, School of Mathematics and Statistics, NSW 2006, Australia, email: \texttt{zhou.zhou@sydney.edu.au}.}
}
\date{\today}
\begin{document}
\maketitle

\begin{abstract}
For an infinite-horizon continuous-time optimal stopping problem under non-exponential discounting, we look for an {\it optimal equilibrium}, which generates larger values than any other equilibrium does on the {\it entire} state space. When the discount function is log sub-additive and the state process is one-dimensional, an optimal equilibrium is constructed in a specific form, under appropriate regularity and integrability conditions. While there may exist other optimal equilibria, we show that they can differ from the constructed one in very limited ways. This leads to a sufficient condition for the uniqueness of optimal equilibria, up to some closedness condition. To illustrate our theoretic results, comprehensive analysis is carried out for three specific stopping problems, concerning asset liquidation and real options valuation. For each one of them, an optimal equilibrium is characterized through an explicit formula.
\end{abstract}

\smallskip


\section{Introduction}
{\it Consistent planning} in Strotz \cite{Strotz55} has been a widely-accepted approach for time-inconsistent 
problems, when a decision maker (an agent) lacks sufficient control over his future selves' behavior. It is a two-phase procedure. First, the agent should find out the strategies that he will actually follow over time 
({\it Phase I}). Formulated as subgame perfect Nash equilibira in subsequent literature, these strategies are constantly called {\it equilibria}. Next, the agent needs to, according to \cite[p.173]{Strotz55}, ``find the best plan among those he will actually follow'' ({\it Phase II}).

In finite-horizon discrete-time models, an equilibrium can be found through straightforward backward sequential optimization, as detailed in Pollak \cite{Pollak68}. It is in fact the {\it unique} equilibrium, as observed in \cite{Kocher96} and \cite{Asheim97}. {\it Phase I} is consequently done, and there is no need of {\it Phase II}. By contrast, once we deviate from either the finite horizon or the discrete-time structure, consistent planning becomes highly nontrivial. 
In infinite-horizon discrete-time settings, ad hoc constructions are used to prove the existence of a few specific equilibria; see e.g. \cite{PP68}, \cite{PY73}, and \cite{GW07}. In continuous-time settings, equilibria are characterized through a system of nonlinear equations, proposed in \cite{EP08} and \cite{BKM17}. This has aroused vibrant research in mathematical finance, such as \cite{EMP12}, \cite{HJZ12}, \cite{Yong12}, and \cite{BMZ14}, among others. Solving this system of equations is, however, difficult; even when it is solved (as in the special cases in \cite{EP08} and \cite{BKM17}), we only obtain one particular equilibrium. In other words, as it currently stands, {\it Phase I} is only partially done, with finding {\it all} equilibria a remaining challenge. {\it Phase II}, on the other hand, has not been addressed at all. It is worth noting that, as opposed to finite-horizon discrete-time models, multiple equilibria may exist, as shown in \cite{PP68}, \cite{Asheim97}, \cite{Kocher96}, and \cite{EP08}. How to select one appropriate equilibrium to use is then a genuine problem. 

Huang and Nguyen-Huu \cite{HN17} recently developed an iterative approach for {\it Phase I} of consistent planning. For stopping problems, \cite{HN17} constructed a fixed-point iteration that can easily generate a large class of equilibria. When the state process is one-dimensional, Huang, Nguyen-Huu, and Zhou \cite{HNZ17} further established that {\it every} stopping strategy converges to an equilibrium via the fixed-point iteration, which facilitates the search for {\it all} equilibria. 
This approach has been applied to stopping problems under non-exponential discounting in \cite{HN17}, and extended to a probability distortion setting in \cite{HNZ17}.

In this paper, we take up the framework in \cite{HN17}, and focus on {\it Phase II} of consistent planning. Specifically, we investigate an infinite-horizon continuous-time stopping problem in which  an agent maximizes his expected payoff, under non-exponential discounting, by choosing an appropriate time to stop a one-dimensional continuous Markov process $X$, taking values in an interval $\X\subseteq \R$.   

An optimality criterion for equilibria, demanded by {\it Phase II}, has hardly been explored in the literature. Here, we propose that an equilibrium is {\it optimal} if it generates a larger value than any other equilibrium does, at {\it every}  
state $x\in\X$; see Definition~\ref{def:optimal equilibrium}. This criterion is rather strong: in economic terms, it requires a subgame perfect Nash equilibrium to be dominant on the entire state space, a rare find in game theory. We, nonetheless, establish the existence of an optimal equilibrium $R^*$, given in the specific form \eqref{R*}; see Theorem~\ref{t1}, which is the main result of this paper.

Theorem~\ref{t1} hinges on several crucial conditions. First, the discount function is required to be log sub-additive, i.e. satisfy \eqref{e1} below. This in particular captures {\it decreasing impatience}, a well-documented feature of empirical discounting in Behavioral Economics and Finance; see the discussion below \eqref{e1} for details. Second, a ``diffusive condition'', i.e. Assumption~\ref{a1}, is imposed on the state process $X$. It serves to ensure that whenever $X$ reaches the boundary of a Borel subset $R$ of $\X$, it is {\it diffusive} enough to enter $R$ immediately. This allows us to restrict our attention to only {\it closed} equilibria (see Section~\ref{subsec:closed equilibria}), with no need to deal with other equilibria of possibly pathological forms. A large class of diffusion processes are shown to satisfy Assumption~\ref{a1}; see Remark~\ref{rem:Assumption satisfied}. Finally, we need \eqref{e34} and \eqref{e32} as the appropriate long-run limiting and integrability conditions for our stopping problem. When \eqref{e32} is violated, an optimal equilibrium in general may not exist; see Proposition~\ref{prop:nu in between} and Remark~\ref{rem:no optimal E} for details.    

A natural ensuing question is whether the optimal equilibrium $R^*$, established in Theorem~\ref{t1}, is actually the unique one. This is {\it not} true in general: in Example~\ref{eg:counterexample},  we construct explicitly $R^*$ and another optimal equilibrium $T^*$, with $R^*$ properly contained in $T^*$. Despite the possibility of multiple optimal equilibria, if we impose closedness requirement for equilibria, we observe that each closed optimal equilibrium can only differ from $R^*$ is very limited ways, as detailed in Corollary~\ref{coro:T* and R*}. Moreover, a reasonable sufficient condition is proposed in Theorem~\ref{t2} for $R^*$ to be the unique closed optimal equilibrium.  This condition can be verified in many practical stopping problems, including those studied in Section~\ref{sec:examples}. 

To illustrate our theoretic results, comprehensive analysis is carried out in Section~\ref{sec:examples} for three specific examples: (i) the stopping of a Bessel process, (ii) the stopping of a geometric Brownian motion, and (iii) an American put option written on a geometric Brownian motion. Each of these examples describes a problem of real options valuation or asset liquidation, and the corresponding optimal equilibrium $R^*$ is characterized through an explicit formula. Another contribution of the analysis here is that it shows how $R^*$ can be found and characterized, even when some of the conditions in Theorem~\ref{t1} actually fail. Indeed, as shown in Section~\ref{subsec:GBM}, one can introduce a related auxiliary stopping problem for which Theorem~\ref{t1} is applicable. 

It is worth noting that in a discrete-time setting, Huang and Zhou \cite{HZ17-discrete} also establish the existence of an optimal equilibrium, for an infinite-horizon stopping problem. Compared with \cite{HZ17-discrete}, our current continuous-time setup is much less amenable. For instance, there is no need in \cite{HZ17-discrete} to carefully differentiate closed equilibria from other equilibria, as every equilibrium is closed in discrete time, under mild continuity assumptions. Moreover, the discrete-time structure leads to the uniqueness of optimal equilibria, which saves \cite{HZ17-discrete} all the endeavor we make in Section~\ref{sec:uniqueness}.
For a detailed comparison between the present paper and \cite{HZ17-discrete}, see Section~\ref{subsec:comparison}.

The paper is organized as follows. Section~\ref{sec:setup} introduces the setup of our stopping problem, the formulation of equilibria, and the optimality criterion for equilibria. Section~\ref{sec:first result} presents a useful basic result, relying only on condition \eqref{e1}, which readily reveals the form of an optimal equilibrium. Section~\ref{sec:existence} establishes the existence of an optimal equilibrium $R^*$. It is shown in Propositions~\ref{p2} and \ref{p2'} that a countable intersection of closed equilibria is again an equilibrium, which is the key to deriving the main result Theorem~\ref{t1}. Section~\ref{sec:uniqueness} studies how another optimal equilibrium $T^*$ can differ from $R^*$, and proposes a sufficient condition, in Theorem~\ref{t2}, for $R^*$ being the unique closed optimal equilibrium. Section~\ref{sec:examples} contains detailed analysis of three specific stopping problems.


\section{The Setup and Preliminaries}\label{sec:setup}
Consider a probability space $(\Omega,\F,\P)$ that supports a time-homogeneous 
continuous strong Markov process $X=(X_t)_{t\in[0,\infty)}$, taking values in some interval $\X\subseteq\R$. Let $\B(\X)$ denote the collection of Borel sets in $\X$.  Let $\mathbb{F}:= (\F_t)_{t\ge 0}$ be the $\P$-augmentation of the filtration generated by $X$, and $\T$ be the collection of all $\mathbb F$-stopping times. For each $x\in\X$, we denote by $\P^x$ (resp. $\E^x$) the probability (resp. expectation) conditioned on $X_0 = x\in\X$. Also, we write $X^x$ for $X$ to emphasize its initial value $X_0 = x$.

Consider a payoff function $f:\,\X\to[0,\infty)$, assumed to be continuous. Also consider a discount function $\delta:\,[0,\infty)\to(0,1]$, assumed to be non-increasing, continuous, and satisfy $\delta(0)=1$ and $\delta(t)\to 0$ as $t\to\infty$. 
The classical optimal stopping problem is formulated as
\begin{equation}\label{classical}
\sup_{\tau\in\T} \E^x[\delta(\tau) f(X_\tau)].  
\end{equation}
Here, we allow $\tau\in\T$ to take the value $\infty$: we take $\delta(\tau) f(X^x_\tau) := \limsup_{t\to\infty} \delta(t) f(X^x_t)$ on $\{\tau = \infty\}$; this is in line with Karatzas and Shreve \cite[Appendix D]{KS-book-98}. 
It is well-known from standard literature (e.g. \cite[Appendix D]{KS-book-98} and Shiryayev \cite{Shiryayev-book-78}) that under fairly general conditions, for any initial state $x\in\X$, there exists an optimal stopping time $\widetilde\tau_x\in \T$ 
for \eqref{classical}. It is then natural to ask whether optimal stopping times obtained at different moments, such as $\widetilde\tau_x$ at time $0$ and $\widetilde\tau_{X^x_t}$ at time $t>0$, are consistent with each other. Specifically, \eqref{classical} is said to be {\it time-consistent} if for any $x\in\X$ and $t>0$, we have
\begin{equation}\label{time-consistency}
\widetilde\tau_x(\omega) = t + \widetilde\tau_{X^x_t}(\omega)\quad \hbox{for a.e. $\omega\in \{\widetilde\tau_{x}\ge t\}$}.
\end{equation}
This condition, in general, {\it only} holds for $\delta(t) := e^{-\alpha t}$ for some $\alpha>0$. When $\delta$ is a general non-exponential discount function, Huang and Nguyen-Huu \cite[Section 2]{HN17} analyze in detail the involved time inconsistency. Intuitively, the failure of \eqref{time-consistency} means that the optimal stopping time we find today, $\widetilde\tau_x$, need not be optimal in the future. Our future self at time $t>0$ is tempted to employ $\widetilde\tau_{X^x_t}$, optimal to him at time $t$, instead of sticking with the previously-determined $\widetilde\tau_x$. This is problematic: although a maximizer $\widetilde\tau_x$ of \eqref{classical} can be found, it will not be carried out by our future selves. In the end, the supremum in \eqref{classical} may not be attained, due to time inconsistency. 

Throughout this paper, we will assume that the discount function $\delta$ satisfies 
\begin{equation}\label{e1}
\delta(s)\delta(t)\leq\delta(s+t),\quad\forall\,s,t>0.
\end{equation}
This particularly captures {\it decreasing impatience}, a well-documented feature of empirical discounting in Behavioral Economics and Finance; see e.g. \cite{Thaler81}, \cite{LT89}, and \cite{LP92}. As detailed in \cite{prelec2004decreasing}, decreasing impatience amounts to steeper discounting for time intervals closer to the present. Note that \eqref{e1} covers non-exponential discount functions that induce decreasing impatience, such as $\delta(t):=1/(1+\beta t)$ for $\beta>0$ (hyperbolic), $\delta(t):=1/(1+\beta t)^k$ for $\beta, k >0$ (generalized hyperbolic), and $\delta(t):=\lambda e^{-r_1 t}+ (1-\lambda) e^{-r_2 t}$ for $r_1, r_2>0$ and $\lambda\in (0,1)$ (pseudo-exponential). For a detailed exposition, see the discussion below \cite[Assumption 3.12]{HN17}. 

As mentioned in Introduction, Strotz \cite{Strotz55} proposes {\it consistent planning} as a way to resolve time inconsistency: an agent should take into account his future selves' behavior, and find a stopping strategy that once being enforced, none of his future selves would want to deviate from. Such strategies, called {\it equilibria} in the literature, will be precisely formulated below.


\subsection{Formulation of Equilibria}

As observed in \cite{HN17}, under time inconsistency, since one may re-evaluate and change his choice of stopping times over time, his stopping strategy is not a single stopping time, but a stopping policy defined below.

\begin{definition}
A Borel measurable $\tau:\X\to \{0,1\}$ is called a stopping policy. 
\end{definition}

Given any current state $x\in \X$, a stopping policy $\tau$ governs when the agent stops: he stops at the first time $\tau(X^x_t)$ yields the value 0. Note that this definition, taken from \cite[Definition 3.1]{HN17}, can be equivalently stated using Borel sets in $\X$. Indeed, by definition,  
\begin{equation}\label{tau and R}
\tau:\X\to \{0,1\}\ \hbox{is a stopping policy}\quad \iff\quad \tau(x) = 1_{R^c}(x)\ \hbox{for some $R\in\B(\X)$}.  
\end{equation}
Here, $\tau$ and $R$ admit the relation $R = \{x\in\X : \tau(x)=0\}$. Hence, the game-theoretic formulation in \cite[Section 3.1]{HN17}, stated in terms of stopping policies $\tau$, can be recast using Borel sets $R$.  

Specifically, to resolve time inconsistency, the agent 
needs to take into account his future selves' behavior, and find the best response to that. Suppose the agent initially planned to take $R\in\B(\X)$ as his stopping policy. Now, at any state $x\in\X$, the agent carries out the game-theoretic reasoning: ``assuming that all my future selves will follow $R\in\B(\X)$, what is the best stopping strategy today in response to that?'' The agent today has only two possible actions: stopping and continuation. If he stops, he gets $f(x)$ immediately. If he continues, given that all his future selves will follow $R\in\B(\X)$, he will eventually stop at the
moment
\[
\rho(x,R):=\inf\{t>0:\ X_t^x\in R\}\in\T.
\]
This leads to the expected payoff 
\begin{equation}\label{J}
J(x, R) := \E^x[\delta(\rho(x, R)) f(X_{\rho(x, R)})].  
\end{equation}
By comparing the payoffs $f(x)$ and $J(x, R)$, we obtain the best stopping strategy for today 
\begin{equation}\label{Theta}
\Theta(R):=S_R\cup \left(I_R \cap R\right)\in \B(\X),
\end{equation}
where 
\begin{equation}\label{regions}
\begin{split}
S_R &:= \{x\in\X: f(x)> J(x,R)\},\\
I_R & := \{x\in\X: f(x)=J(x,R)\},\\
C_R &:= \{x\in\X: f(x)< J(x,R)\}.   
\end{split}
\end{equation}
Here, $S_R$, $I_R$, and $C_R$ are called the {\it stopping region}, the {\it indifference region}, and the {\it continuation region}, respectively. In particular, on $I_R$, the agent is indifferent between stopping and continuation as they yield the same payoff. There is then no incentive for the agent to deviate from the original stopping strategy $R\in\B(\X)$. This gives rise to the term $I_R \cap R$ in \eqref{Theta}. 

\begin{definition}\label{def:E}
$R\in\mathcal{B}(\X)$ is called an equilibrium if $\Theta(R)=R.$
We denote by $\cE$ the set of all equilibria.
\end{definition}

\begin{remark}
Under current context where $X$ is time-homogeneous, \eqref{Theta} is simply a re-formulation of (3.6) in \cite{HN17}, using the relation \eqref{tau and R}. Consequently, Definition~\ref{def:E} is simply a re-statement of \cite[Definition 3.7]{HN17}, thanks again to \eqref{tau and R}. 
\end{remark}

\begin{remark} [Existence of an equilibrium] \label{rem:trivial equilibrium}
The entire state space $\X$ is an equilibrium. Indeed, for any $x\in\X$, since $\rho(x,\X) = 0$, we have $J(x,\X)=f(x)$. This implies $I_\X = \X$, and thus $\Theta(\X)=\X$. 
\end{remark}

The general methodology for finding equilibria is the fixed-point iteration introduced in \cite{HN17}: one starts with an arbitrary $R\in\B(\X)$, and apply $\Theta$ to it repetitively until an equilibrium is reached. The next result is a direct consequence of \cite[Theorem 3.16]{HN17} and \eqref{tau and R}. 

\begin{proposition}\label{prop:iteration}
For any $R\in \B(\X)$ such that $R\subseteq \Theta(R)$, we have $\Theta^n(R)\subseteq \Theta^{n+1}(R)$ for all $n\in\N$. Moreover,
\[
R_0 := \lim_{n\to\infty} \Theta^n(R) = \bigcup_{n\in\N}\Theta^n(R)
\]
belongs to $\cE$. 
\end{proposition}


\subsection{Optimality of an Equilibrium}

Finding equilibria is only the first phase of {\it consistent planning} in Strotz \cite{Strotz55}. In the second phase, the agent should choose the {\it best} one among all equilibria. 
This requires certain optimality criterion for an equilibrium, which has not been addressed in the literature. 

For each $R\in\cE$, define the associated value function by 
\[
V(x,R):=f(x)\vee J(x,R),\quad \hbox{for all}\ x\in\X. 
\] 

\begin{definition}\label{def:optimal equilibrium}
$\hat R\in\cE$ is called an optimal equilibrium, if for any $R\in\cE$,
$$V(x,\hat R)\geq V(x,R)\quad \forall x\in\X.$$
\end{definition}
\noindent This is a strong criterion for optimality, as it requires dominance of $\hat R$ over $R$ at {\it every} $x\in\X$. Another way to interpret Definition~\ref{def:optimal equilibrium} is that, similarly to classical optimal stopping in \eqref{classical}, we aim to solve
\begin{equation}\label{classical'}
\sup_{R\in\cE} V(x,R). 
\end{equation}
What differs from classical optimal stopping is that we are not satisfied with solving \eqref{classical'} for each $x\in\X$. Instead, we intend to find one single $\hat R\in\cE$ that solves \eqref{classical'} simultaneously for all $x\in\X$. This uniform dominance on the entire space $\X$ is necessary to avoid another level of time inconsistency.  

The main goal of this paper is to show that, while the optimality criterion is rather strong, an optimal equilibrium does exist.


\section{The First Result}\label{sec:first result}
We observe that \eqref{e1} already implies that an $R\in\cE$ generates larger values than any $T\in\B(\X)$ containing $R$.

\begin{lemma}\label{t3}
For any $R\in\cE$ and $T\in\B(\X)$ with $R\subseteq T$, 
\begin{equation}\label{ee4}
J(x,R)\geq J(x,T)\quad\forall\,x\in\mathbb{X}.
\end{equation}
\end{lemma}

\begin{proof}
Fix $x\in\X$. For simplicity, we will write $v =\rho(x,R)$ and $\tau =\rho(x,T)$. With $R\subseteq T$, we have $\tau \le v$. Consider the event
$
A := \{\omega\in\Omega : \tau < v\}. 
$ 
Observe that
\begin{align}
J(x,R) &= \E^x\left[\delta(v)f(X_v) 1_A\right] +\E^x\left[\delta(v)f(X_v) 1_{A^c}\right]\notag\\
&= \E^x\left[\E^x\left[\delta(v)f(X_v)\mid\F_\tau\right] 1_A \right] +\E^x\left[\delta(\tau)f(X_\tau) 1_{A^c}\right]\notag\\
&\ge \E^x \left[\delta(\tau)\E^x\left[\delta(v-\tau)f(X_v) \mid\F_\tau\right] 1_A \right] +\E^x\left[\delta(\tau)f(X_\tau) 1_{A^c}\right],\label{ee1}
\end{align}
where the inequality follows from \eqref{e1} and $f$ being nonnegative. The strong Markov property of $X$ implies that
\begin{equation}\label{ee2}
\E^x\left[\delta(v-\tau)f\left(X_v\right)\mid \mathcal{F}_{\tau}\right] 1_A
=\E^{X_{\tau}^x}\left[\delta(\rho(X_{\tau}^x,R))f\left(X_{\rho(X_{\tau}^x,R)}\right)\right] 1_A = J(X_{\tau}^x,R) 1_A.
\end{equation}
Also note that $X_{\tau}^x\notin R$ on the event $A$. Since $R$ is an equilibrium, we deduce from \eqref{Theta} that
\begin{equation}\label{ee3}
f(X_{\tau}^x) \le J(X_{\tau}^x,R)\quad \hbox{on $A$}.
\end{equation}
By \eqref{ee1}, \eqref{ee2}, and \eqref{ee3}, we conclude that
\begin{align*}
J(x,R) &\ge \E^x\left[\delta(\tau)J(X_{\tau},R) 1_A\right]+\E^x\left[\delta(\tau)f(X_\tau) 1_{A^c}\right]\\ 
&\ge \E^x[\delta(\tau) f(X_\tau) 1_A]+\E^x\left[\delta(\tau)f(X_\tau) 1_{A^c}\right] = J(x,T).
\end{align*}
\end{proof}

Lemma~\ref{t3} leads to the intriguing observation: for equilibria, the smaller the better.

\begin{corollary}\label{coro:smaller better}
For any $R$, $T\in\cE$ with $R\subseteq T$, $V(x,R)\ge V(x,T)$ for all $x\in\X$. 
\end{corollary}

At first glance, it is intuitively puzzling how an equilibrium $R$ can possibly dominate {\it any} larger equilibrium that contains $R$, on the {\it entire} state space. In fact, Corollary~\ref{coro:smaller better} simply reflects decreasing impatience in decision making (recall the discussion below \eqref{e1}). 

To see this, consider two rewards at time $t$ and time $t+s$, respectively, with $t,s> 0$. If the agent at time $t$ prefers the second reward over the first one, note that this time-$t$ preference is kept, and actually amplified, at time 0, under decreasing impatience. Recall that decreasing impatience amounts to steeper discounting for time intervals closer to the present. To the agent at time $t$, the interval under consideration $[t,t+s]$ is essentially $[0,s]$, so he discounts steeply on this imminent interval when comparing the two rewards. By contrast, to the agent at time 0, the interval $[t,t+s]$ is further down in the future, so he discounts less steeply on this interval when comparing the two rewards. Hence, if the second reward is more valuable than the first one under steeper discounting (i.e. in the eyes of the agent at time $t$), it must be more valuable than the first one, to a greater extent, under flatter discounting (i.e. in the view of the agent at time 0). 



Now, take two equilibria $R$ and $T$. If $R\subseteq T$, $\rho(x,T)\le \rho(x,R)$ for all $x\in\X$. For $x\in R$, since $\rho(x,T)= \rho(x,R)=0$, we have $f(X^x_{\rho(x,T)})=f(X^x_{\rho(x,R)})=x$. The agent at the state $x\in R$ is then indifferent between $R$ and $T$, as they yield the same immediate payoff. For $x\notin R$, the fact that $R$ is an equilibrium indicates $f(X^x_{\rho(x,T)})\le J(X^x_{\rho(x,T)},   R)= \E^{y}[\delta(\rho(y,R))f(X_{\rho(y,R)})]$, with $y:= X^x_{\rho(x,T)}$. That is, to the agent at the state $y= X^x_{\rho(x,T)}$, the reward obtained from reaching $R$ (i.e. the second reward) is more valuable than that obtained from reaching $T$ (i.e. the first reward). As discussed above, decreasing impatience stipulates that the agent at the initial state $x\notin R$ has the same preference, i.e. $R$ is preferred over $T$ at $x\notin R$. Therefore, we conclude that the smaller equilibrium $R$ is preferred over (or at least no worse than) the larger one $T$, at {\it every} $x\in \X$.


A candidate optimal equilibrium can already be deduced from Lemma~\ref{t3}.  

\begin{proposition}\label{prop:first result}
If
$
\tilde R := \bigcap_{R\in\cE} R
$
is an equilibrium, then it is an optimal equilibrium.
\end{proposition}

\begin{proof}
By definition $\tilde R \subseteq R$ for all $R\in\cE$. If $\tilde R$ is an equilibrium, we can conclude from Lemma~\ref{t3} that for any $R\in\cE$, $J(x,\tilde R) \ge J(x,R)$, and thus $V(x,\tilde R) \ge V(x,R)$, for all $x\in\X$. Thus, $\tilde R$ is an optimal equilibrium. 
\end{proof}

For $\tilde R:= \bigcap_{R\in\cE} R$ to be an equilibrium, two technical issues have to be resolved: (i) Is the intersection of two equilibria again an equilibrium? (ii) If (i) is true, can it be generalized to an uncountable intersection of equilibria? In particular, how do we ensure that $\tilde R$ is Borel, given that it is an uncountable intersection of Borel sets? Section~\ref{sec:existence} below is devoted to these questions, and eventually establishes the existence of an optimal equilibrium.   

Before heading to the detailed analysis in Section~\ref{sec:existence}, we point out a straightforward consequence of Proposition~\ref{prop:first result} that is already useful for some stopping problems. 

\begin{corollary}\label{coro:emptyset is optimal}
If $\emptyset$ is an equilibrium, then it is an optimal equilibrium. 
\end{corollary}

\begin{example}\label{eg:nu>0}
Let $X$ be a geometric Brownian motion, given by
\begin{equation}\label{GBM}
dX_t = \mu X_t dt + \sigma X_t dW_t,
\end{equation}
where $\mu\in\R$ and $\sigma>0$ are constants, and $W$ is a standard Brownian motion. The state space is $\X=(0,\infty)$. Consider the hyperbolic discount function
\begin{equation}\label{hyperbolic}
\delta(t) := \frac{1}{1+\beta t}\quad t\ge0,
\end{equation}
where $\beta>0$ is a constant, as well as the payoff function $f(x) := x$ for $x>0$. Under current setting, \eqref{J} becomes
\begin{equation}\label{J GBM}
J(x,R) = \E^x\left[\frac{X_{\rho(x,R)}}{1+\beta\rho(x,R)}\right],\quad \hbox{for}\ x>0,\ R\in\B(\X). 
\end{equation}
Define 
\begin{equation}\label{nu}
\nu := \frac{\mu}{\sigma^2}-\frac12,
\end{equation}
and recall that for any $x>0$, 
\begin{equation} \label{X(infty)}
\lim_{t\to\infty} \frac{X^x_t}{1+\beta t} = \lim_{t\to\infty} \frac{x \exp{\left(\nu\sigma^2t + \sigma W(t)\right)}}{1+\beta t} = 
\begin{cases}
\infty,\quad  &\hbox{if $\nu>0$};\\
0,\quad  &\hbox{if $\nu\le 0$}.
\end{cases}
\end{equation}
For $\nu>0$, \eqref{X(infty)} implies $J(x,\emptyset) = \infty > f(x)$ for all $x>0$. This shows that $\emptyset\in \cE$ and $V(x,\emptyset)=\infty$ for all $x>0$. Thus, $\emptyset$ is an optimal equilibrium, which confirms Corollary~\ref{coro:emptyset is optimal}. 
\end{example}


\section{Existence of an Optimal Equilibrium}\label{sec:existence}

For each $x\in\X$, consider the running maximum $\overline X$ and the running minimum $\underline X$ of the state process $X$, i.e.
\[
\overline X^x_t := \max_{s\in[0,t]} X^x_s\qquad \hbox{and}\qquad \underline X^x_t := \min_{s\in[0,t]} X^x_s,\qquad t\ge 0.
\]
We impose the following condition on $X$.

\begin{assumption}\label{a1}
For any $x\in\X$, $\P^x[\overline X_t > x] = \P^x[\underline X_t < x] =1$ for all $t> 0$. 
\end{assumption}

\begin{remark}\label{rem:T^x_x=0}
For any $x\in\X$, consider $T^x_x := \inf\{t > 0 : X^x_t=x\}$, the first revisit time to the initial value $x$. 
A direct consequence of Assumption \ref{a1} is that $T^x_x=0$ $\P^x$-a.s.
\end{remark}

The purpose of Assumption~\ref{a1}, intuitively speaking, is to ensure that ``whenever $X$ touches the boundary of $R$, for any $R\in\B(\X)$, it will enter $R$ immediately''. This property will play a crucial role in Section~\ref{subsec:closed equilibria}. Note that Assumption~\ref{a1} is not very restrictive, since it already covers a large class of diffusion processes, as explained below.

\begin{remark}\label{rem:Assumption satisfied}
A one-dimensional Brownian motion obviously satisfies Assumption~\ref{a1}; see e.g. Problem 7.18 on p. 94 of \cite{KS-book-91}. This can in fact be generalized to a much broader class of diffusion processes. Specifically, let $X$ satisfy the dynamics
$$dX_t=\mu(X_t)dt+\sigma(X_t)dB_t,$$
where $B$ is a standard one-dimensional Brownian motion, and $\mu,\sigma$ are real-valued deterministic functions such that $\sigma(\cdot)>0$ and $\int_0^t \theta^2(X^x_s) ds <\infty$ $\P^x$-a.s. for all $x\in\X$ and $t\ge 0$, where $\theta(\cdot):=\mu(\cdot)/\sigma(\cdot)$. If the process
$$Z_t:=\exp{\left(-\int_0^t\theta(X_s)dB_s-\frac{1}{2}\int_0^t\theta^2(X_s)ds\right)},\quad t\ge 0$$
is a martingale, \cite[Lemma 3.1]{HNZ17} shows that $X$ satisfies Assumption~\ref{a1}. 
\end{remark}


\subsection{Closed Equilibria}\label{subsec:closed equilibria}
In this subsection, we will show that Assumption~\ref{a1} allows us to restrict our attention to only {\it closed} equilibria. For any $R\in\B(\X)$,  $\overline R$ denotes the closure of $R$. 

\begin{lemma}\label{lem:rho R = rho bar R}
Suppose \asref{a1} holds.  For any $R\in\B(\X)$, 
\[
\rho(x,R) = \rho(x,\overline R)\quad \P^x\hbox{-a.s.}\quad \hbox{$\forall x\in\X$.}
\] 
In particular, $\rho(x,R) = \rho(x,\overline R)=0$ $\P^x$-a.s. for all $x\in\overline R$.
\end{lemma}

\begin{proof}
Fix $x\in\X$. If $x$ is an interior point of $R$, by definition $\rho(x,R)=\rho(x,\overline R)=0$. 
If $x\in \partial R$, there are two cases.
\begin{itemize}[leftmargin=*]
\item {\bf Case I:} There exist $\{x_n\}_{n\in\N}$ in $R$ such that $x_n\to x$. For $\P^x$-a.e. $\omega\in\Omega$, \asref{a1} implies that $\{X^x_s(\omega) : s\in[0,\eps]\}$ must intersect $\{x_n\}_{n\in\N}$, for all $\eps>0$. This implies $\rho(x,R)=0$ $\P^x$-a.s. Since $\rho(x,\overline R)\le \rho(x,R)$ by definition, we get  $\rho(x,\overline R)=\rho(x,R)=0$ $\P^x$-a.s. 
\item {\bf Cases II:} $x$ is an isolated point of $R$, i.e. $x\in R$ and there exists $\eps>0$ such that $(x-\eps,x+\eps)\cap R = \{x\}$. By Remark~\ref{rem:T^x_x=0}, $T^x_x = 0$ $\P^x$-a.s. It follows that $\rho(x,R)=0$, and thus $\rho(x,\overline R) =0$, $\P^x$-a.s. 
\end{itemize}
Finally, if $x\notin \overline R$, consider the random variable 
$Y:= X^x_{\rho(x,\overline R)}.$ 
For any $\omega\in\Omega$, since $X$ is a continuous process, $Y(\omega)\in\partial R$. From the analysis above, we have $\P^{Y(\omega)}[\rho(Y(\omega),R)=0] =1$. By the strong Markov property of $X$, 
\[
\P^x\left[\rho(X_{\rho(x,\overline R)},R)=0\ \middle|\ \F_{\rho(x,\overline R)}\right] (\omega) = \P^{Y(\omega)}[\rho(Y(\omega),R)=0] =1,\quad \hbox{for $\P^x$-a.e. $\omega\in\Omega$}.
\]
Taking expectation on both sides yields $\P^x[\rho(X_{\rho(x,\overline R)},R)=0]=1$. It follows that $\rho(x,R)(\omega)= \rho(x,\overline R) (\omega) + \rho(X^x_{\rho(x,\overline R)}, R)(\omega) =  \rho(x,\overline R) (\omega)$ for $\P^x$-a.e. $\omega\in\Omega$. 
\end{proof}

Lemma~\ref{lem:rho R = rho bar R} admits several useful implications.

\begin{remark}\label{rem:value unchanged}
Under \asref{a1}, for any $R\in\B(\X)$, since $\rho(x,R) = \rho(x,\overline R)$ $\P^x$-a.s. $\forall x\in\X$, by definition 
$J(x,R) = J(x,\overline R)$, and thus $V(x,R) = V(x,\overline R)$, $\forall x\in\X$. Moreover, with $J(x,R) = J(x,\overline R)$ for all $x\in\X$, \eqref{regions} shows that $S_R=S_{\overline R}$, $I_R=I_{\overline R}$, and $C_R=C_{\overline R}$.
\end{remark}

\begin{remark}\label{rem:R subset Theta R}
Under \asref{a1}, observe that $R\subseteq\Theta(R)$, for all $R\in\B(\X)$. Indeed, for any $x\in R$, since $\rho(x,R)=0$ $\P^x$-a.s. (Lemma~\ref{lem:rho R = rho bar R}), we have $J(x,R)=f(x)$, i.e. $x\in I_R$. It follows that $R\subseteq I_R$, and thus
\begin{equation}\label{Theta'}
\Theta(R)=S_R\cup \left(I_R \cap R\right) = S_R\cup R. 
\end{equation}
This, together with Proposition~\ref{prop:iteration}, gives a desirable result for finding equilibria: under \asref{a1}, 
every $R\in\B(\X)$ converges to an equilibrium via the fixed-point iteration.
Specifically, 
\[
\cE = \bigg\{\lim_{n\to\infty} \Theta^n(R) : R\in \B(\X) \bigg\}. 
\]
\end{remark}

\begin{lemma}\label{lem:bar R in E}
Suppose \asref{a1} holds. 
\begin{itemize}
\item [(i)] For any $R\in\cE$, $S_R = \emptyset$. 
\item [(ii)] For any $R\in\B(\X)$, $R\in\cE$ if and only if $\overline R\in\cE$. 
\end{itemize} 
\end{lemma}

\begin{proof}
(i) Fix $R\in\cE$. For any $x\in\X$, if $x\in\overline R$, then $\rho(x,R) =0$ $\P^x$-a.s. (Lemma~\ref{lem:rho R = rho bar R}), which implies $J(x,R)=f(x)$, i.e. $x\in I_R$. If $x\notin\overline R$, then $x\notin R = \Theta(R)= S_R\cup R$, thanks to $R\in\cE$ and \eqref{Theta'}. It follows that $S_R = \emptyset$. 

(ii) For any $R\in\cE$, by Remark~\ref{rem:value unchanged} and (i), $S_{\overline R}=S_R=\emptyset$. In view of \eqref{Theta'}, $\Theta(\overline R) = S_{\overline R}\cup  \overline{R}= \overline{R}$, i.e. $\overline R\in\cE$. By switching the roles of $R$ and $\overline R$, we can prove that $\overline R\in\cE$ implies $R\in\cE$. 
\end{proof}

Lemma~\ref{lem:bar R in E} (ii) and Remark~\ref{rem:value unchanged} convey an important message: to compare different equilibria in terms of their associated values, it suffices to restrict our attention to {\it closed} equilibria. After all, the closure of $R\in\cE$ remains an equilibrium, with the {\it same} values. In fact, $R$ and $\overline R$ induce the same stopping behavior, as $S_R$, $I_R$, and $C_R$ in \eqref{regions} stay intact after we take the closure of $R$. 

This additional closedness condition is the key to improving an equilibrium in a systematic fashion, as we will introduce now.


\subsection{Intersections of Equilibria}

Under appropriate integrability conditions, we show that the intersection of two closed equilibria is again an equilibrium, with {\it larger} values. 

\begin{proposition}\label{p2}
Suppose \asref{a1} holds. Assume additionally that for each $x\in\X$, 
\begin{equation}\label{e34}
\delta(t)f(X^x_t)\rightarrow 0\quad \hbox{as}\ t\rightarrow\infty\qquad \P^x\hbox{-a.s.},
\end{equation}
and 
\begin{equation}\label{e32}
\E^x\bigg[\sup_{t\in[0,\infty)}\delta(t)f(X_t)\bigg]<\infty.
\end{equation}
Then, for any $R$, $T\in\cE$ that are closed, we have
\begin{equation}\label{e6}
J(x,R\cap T)\geq J(x,R)\vee J(x,T),\quad\forall\,x\in\X.
\end{equation}
Hence, $R\cap T$ also belongs to $\cE$.
\end{proposition}

\begin{proof}
Fix $x\in\X$. Throughout this proof, for simplicity of notation, we define, for all $n=0,1,\dotso$, the following:
\begin{itemize}
\item [1.] $\tau_0:=0$,\quad $\tau_{2n+1}:=\inf\{t>\tau_{2n}:\ X_t^x\in T\}$,\quad $\tau_{2n+2}:=\inf\{t>\tau_{2n+1}:\ X_t^x\in R\}$;
\item [2.] $A_n:=\{\omega\in\Omega :\ \tau_n(\omega)<\infty,\ X_{\tau_n}^x(\omega)\notin R\cap T\}$.
\end{itemize}
Let $v:=\rho(x,R\cap T)$. By the definition of $A_1$, we have 
\begin{align}
\notag J(x,R\cap T)-J(x,T) &=\E^x\left[1_{A_1}\left(\delta(v)f\left(X_v\right)-\delta(\tau_1)f\left(X_{\tau_1}\right)\right)\right]\\
\notag &=\E^x\left[1_{A_1}\delta(\tau_1)\left(\frac{\delta(v)}{\delta(\tau_1)}f\left(X_v\right)-f\left(X_{\tau_1}\right)\right)\right]\\
\notag&\geq\E^x\left[1_{A_1}\delta(\tau_1)\left(\delta(v-\tau_1)f\left(X_v\right)-f\left(X_{\tau_1}\right)\right)\right]\\
\label{e2}&=\E^x\left[1_{A_1}\delta(\tau_1)\left(\E^x\left[\delta(v-\tau_1)f\left(X_v\right)|\mathcal{F}_{\tau_1}\right]-f\left(X_{\tau_1}\right)\right)\right]
\end{align}
where the inequality follows from \eqref{e1} and $f$ being nonnegative. The strong Markov property of $X$ implies that
\begin{align}
\E^x\left[\delta(v-\tau_1)f\left(X_v\right)\Big|\ \mathcal{F}_{\tau_1}\right] 1_{A_1}
&=\E^{X_{\tau_1}^x}\left[\delta(\rho(X_{\tau_1}^x,R\cap T))f\left(X_{\rho(X_{\tau_1}^x,R\cap T)}\right)\right] 1_{A_1} \nonumber\\
&= J(X_{\tau_1}^x,R\cap T) 1_{A_1}.\label{Mark prop}
\end{align}
On the event $A_1$, note that the closedness of $T$ entails $X_{\tau_1}^x\in T$. Thus,  the fact that $X_{\tau_1}^x(\omega)\notin R\cap T$ implies $X_{\tau_1}^x\notin R$. Since $R$ is an equilibrium, by \eqref{Theta} 
\begin{equation}\label{f<J}
f(X_{\tau_1}^x) \le J(X_{\tau_1}^x,R)\quad \hbox{on}\ A_1.
\end{equation}
With \eqref{Mark prop} and \eqref{f<J}, \eqref{e2} yields
\begin{equation}\label{e3}
J(x,R\cap T)-J(x,T)\geq\E^x\left[1_{A_1}\delta(\tau_1)\left(J(X_{\tau_1},R\cap T)-J(X_{\tau_1},R)\right)\right].
\end{equation}

In the following, we will carry out the above argument recursively. First, repeating the above argument for $J(X_{\tau_1}^x,R\cap T)-J(X_{\tau_1}^x,R)$, instead of $J(x,R\cap T)-J(x,T)$, we obtain 
\begin{align}\label{e4}
J(X_{\tau_1}^x,R\cap T)&-J(X_{\tau_1}^x,R)\nonumber \\
 &\geq \E^{X_{\tau_1}^x}\left[1_{A(X_{\tau_1}^x)}\delta(\rho(X_{\tau_1}^x,R))\left(J(X_{\rho(X_{\tau_1}^x,R)},R\cap T)-J(X_{\rho(X_{\tau_1}^x,R)},T)\right)\right],
\end{align}
where
$$A(y):=\left\{\omega\in\Omega:\ \rho(y,R)(\omega)<\infty,\ X_{\rho(y,R)}^y(\omega)\notin R\cap T\right\},\quad \hbox{for}\ y\in\X.$$
Denote the expectation in \eqref{e4} by $I$. By the strong Markov property of $X$,
$$I\cdot 1_{A_1}=\E^x\left[1_{A_2}\delta(\tau_2-\tau_1)\left(J(X_{\tau_2},R\cap T)-J(X_{\tau_2},T)\right)\ \big|\ \mathcal{F}_{\tau_1}\right]\cdot 1_{A_1}.$$
This, together with \eqref{e3} and \eqref{e4}, implies 
\begin{equation}\notag
J(x,R\cap T)-J(x,T) \geq \E^x\left[1_{A_1}1_{A_2}\delta(\tau_1)\delta(\tau_2-\tau_1)\left(J(X_{\tau_2},R\cap T)-J(X_{\tau_2},T)\right)\right].
\end{equation}
Continuing this procedure, we can get for all $n=1,2,\dotso$,
\begin{align}
\notag &J(x,R\cap T)-J(x,T)\geq \E^x\left[1_{A_1}\dotso 1_{A_n}\delta(\tau_1-\tau_0)\dotso\delta(\tau_n-\tau_{n-1})\left(J(X_{\tau_n},R\cap T)-J(X_{\tau_n},U)\right)\right]
\end{align}
where $U=R$ if $n$ is odd and $U=T$ if $n$ is even. Define
$$\sigma_n:=\inf\{t> \tau_n:\ X_t^x\in R\cap T\}.$$
We can rewrite the previous inequality as follows: for all $n=1,2,\dotso$,
\begin{align}
\notag &\hspace{0.2in}J(x,R\cap T)-J(x,T)\\
\label{e5}&\ge \E^x\left[1_{A_1}\dotso 1_{A_n}\delta(\tau_1-\tau_0)\dotso\delta(\tau_n-\tau_{n-1})\left(\delta(\sigma_n-\tau_n)f(X_{\sigma_n})-\delta(\tau_{n+1}-\tau_n)f(X_{\tau_{n+1}})\right)\right].
\end{align}
Denote by $H_n$ the random variable inside the expectation in \eqref{e5}. We claim that $H_n\rightarrow 0$ $\P^x$-a.s. Let us consider the two distinct cases.
\begin{itemize}[leftmargin=*]
\item \textbf{Case I:} $\lim_{n\to \infty}\tau_n=\infty$. Since $\sigma_n\ge \tau_n$ by definition, $\lim_{n\to \infty}\sigma_n=\infty$. By \eqref{e1} and \eqref{e34},
\[
|H_n|\leq \delta(\sigma_n)f(X_{\sigma_n}^x)+\delta(\tau_{n+1})f(X_{\tau_{n+1}}^x)\rightarrow 0, 
\]
and thus $H_n\rightarrow 0$.
\item \textbf{Case II:} $\lim_{n\rightarrow\infty}\tau_n<\infty$. Set $\tau:=\lim_{n\rightarrow\infty}\tau_n$. Since $X$ is a continuous process, $X_{\tau_n}\rightarrow X_\tau$. For every $n$ even, since $R$ is closed, we have $X_{\tau_{n}}\in R$ by the definition of $\tau_n$. This implies $X_\tau\in R$, again by the closedness of $R$. Similarly, for every $n$ odd, since $T$ is closed, we have $X_{\tau_{n}}\in T$ by the definition of $\tau_n$. This implies $X_\tau\in T$, again by the closedness of $T$. Hence, $X_\tau\in R\cap T$, and thus $\tau_n\leq\sigma_n\leq\tau$. It follows that $\sigma_n\rightarrow\tau$. Thanks to the continuity of $\delta$ and $f$, 
$$\delta(\sigma_n-\tau_n)f(X_{\sigma_n})-\delta(\tau_{n+1}-\tau_n)f(X_{\tau_{n+1}})\ \rightarrow\ f(X_\tau)-f(X_\tau)=0,$$
which implies $H_n\rightarrow 0$.
\end{itemize}
Now, by the fact that $H_n\to 0$ $\P^x$-a.s. and 
$$|H_n|\leq 2\sup_{t\in[0,\infty)}\delta(t)f\left(X_t^x\right),$$
we can apply the dominated convergence theorem to \eqref{e5} as $n\to\infty$, thanks to \eqref{e32}. This yields 
$$J(x,R\cap T)\geq J(x,T).$$
By switching the roles of $R$ and $T$ in the proof above, we can also obtain
$J(x,R\cap T)\geq J(x,R).$ Then \eqref{e6} follows.

It remains to show that $R\cap T$ is an equilibrium. Fix $x\notin R\cap T$. If $x\notin R$, by the fact that $R$ is an equilibrium and \eqref{e6}, $f(x)\le J(x,R)\leq J(x,R\cap T)$. This implies $x\in I_{R\cap T}\cup C_{R\cap T}$. By \eqref{Theta}, $x\notin R\cap T$, and $x\in I_{R\cap T}\cup C_{R\cap T}$, we conclude that $x\notin \Theta(R\cap T)$.
If $x\notin T$, the same argument gives $x\notin\Theta(R\cap T)$. We therefore conclude that $\Theta(R\cap T)\subseteq R\cap T$. Since $R\cap T \subseteq \Theta(R\cap T)$ under Assumption~\ref{a1} (see Remark~\ref{rem:R subset Theta R}), we have $\Theta(R\cap T) = R\cap T$, i.e. $R\cap T\in\cE$.
\end{proof}

\begin{remark}
The closedness of $R$, $T\in\cE$ is indispensable for Proposition~\ref{p2}. Particularly, it is used to establish $X_{\tau_1}^x\notin R$ above \eqref{f<J}, as well as ``$H_n\to 0$'' in Case II of the proof. 

In fact, without the closedness condition, Proposition~\ref{p2} fails in a trivial way. For instance, suppose Assumption~\ref{a1} holds, and the payoff function $f$ is strictly positive and satisfies \eqref{e34} and \eqref{e32}. We can construct two equilibria, without closedness, whose intersection is no longer an equilibrium. Take $R_1 := \X\cap \Q$ and $R_2:=\X\cap \Q^c$. Since $\overline R_1=\overline R_2= \X\in \cE$ (Remark~\ref{rem:trivial equilibrium}), Lemma~\ref{lem:bar R in E} (ii) asserts $R_1\in\cE$ and $R_2\in\cE$. By construction, $R_1\cap R_2 = \emptyset$. Since \eqref{e34} implies $J(x,\emptyset) = 0 < f(x)$ for all $x\in\X$, $R_1\cap R_2 =\emptyset$ is not an equilibrium.    
\end{remark}




The next step is to investigate if a {\it countable} intersection of equilibria is again an equilibrium. To this end, we need the following technical result. 

\begin{lemma}\label{l5}
Let $(S_n)_{n\in\N}$ in $\mathcal{B}(\X)$ be a non-increasing sequence of closed sets. For any $x\in\X$,
$$\lim_{n\rightarrow\infty}\rho(x,S_n)=\rho(x,S_\infty)\ \ \text{$\P^x$-a.s.},\quad\text{with}\ \ S_\infty:=\bigcap_{n\in\N} S_n.$$
\end{lemma}

\begin{proof}
Take an arbitrary $x\in\X$. The ``$\le$'' relation holds trivially, as $S_n\supseteq S_\infty$ for all $n\in\N$. To prove the ``$\ge$'' relation, since it holds trivially when $\lim_{n\rightarrow\infty}\rho(x,S_n)=\infty$, we assume below that $\lim_{n\rightarrow\infty}\rho(x,S_n)<\infty$. Define $\tau_n:=\rho(x,S_n)$ for all $n\in\N$, and set $\tau:=\lim_{n\rightarrow\infty}\tau_n$. If $x\in S_\infty$, $\rho(x,S_n) = \rho(x,S_\infty) = 0$ for all $n\in\N$, and thus the desired result trivially holds. Suppose $x\notin S_\infty$. Fix $\omega\in\Omega$.
For any $m\in\N$, since $X$ is a continuous process, we have $X^x_{\tau_n}(\omega)\rightarrow X^x_\tau(\omega)$, and thus $X^x_{\tau_n}(\omega)\in [X^x_\tau(\omega)-1/m,X^x_\tau(\omega)+1/m]$ for $n$ large enough. It follows that $S_n\cap[X^x_\tau(\omega)-1/m,X^x_\tau(\omega)+1/m]\neq\emptyset$ for $n$ large enough, and thus for all $n$. Now, since $(S_n)_{n\in\N}$ is a non-increasing sequence of closed sets, we can apply Cantor's intersection theorem and obtain
\[
S_\infty\cap[X^x_\tau(\omega)-1/m,X^x_\tau(\omega)+1/m]=\bigcap_{n\in\N} \left(S_n \cap [X^x_\tau(\omega)-1/m,X^x_\tau(\omega)+1/m]\right)\neq\emptyset,\quad \forall m\in\N.
\]
By definition, $S_\infty$ is closed. We can thus apply Cantor's intersection theorem again and get
\[
S_\infty\cap \{X^x_\tau(\omega)\} = \bigcap_{m\in\N} \left(S_\infty \cap [X^x_\tau(\omega)-1/m,X^x_\tau(\omega)+1/m]\right)\neq\emptyset. 
\]
That is, $X^x_\tau(\omega)\in S_\infty$. It follows that $\rho(x,S_\infty)(\omega)\le \tau(\omega)$, as desired. 
\end{proof}

\begin{proposition}\label{p2'}
Suppose \asref{a1}, \eqref{e34}, and \eqref{e32} hold. For any sequence $(R_n)_{n\in\N}$ of closed equilibria, $R:= \bigcap_{n\in\N} R_n$ belongs to $\cE$. Moreover, for each $n\in\N$,  
\begin{equation*}
J(x,R)\geq J(x,R_n),\quad\forall\,x\in\X.
\end{equation*}
\end{proposition}

\begin{proof}
We can rewrite $R$ as follows: 
\[
R=\bigcap_{n\in\N}T_n,\quad \hbox{with}\ \ T_n:=\bigcap_{i=1,2,...,n} R_i,\  n\in\N. 
\]
By \propref{p2}, $T_n\in\cE$ for all $n\in\N$, and $J(x,T_n)$ is non-decreasing in $n$, for each $x\in\X$.
 
For any $x\in R$, since $x\in T_n$ for all $n\in\N$, the fact that every $T_n$ is an equilibrium implies 
$f(x)\geq J(x,T_n),\ \forall n\in\N$.
By \eqref{e32} and \lemref{l5}, we may apply the dominated convergence theorem and get
$f(x)\geq J(x,R)$,
which shows that $x\in S_{R}\cup I_{R}$. Thus, $R\subseteq S_{R}\cup I_{R}$, which implies  $R\subseteq S_{R}\cup (I_{R}\cap R) = \Theta(R)$. On the other hand, for any $x\notin R$, there exists $N\in\N$ such that $x\notin T_n$ for $n\geq N$. Since $T_n\in\cE$ for all $n\ge N$,
$f(x)\le J(x,T_n)\ \hbox{for all}\ n\geq N$.
By the dominated convergence theorem (thanks again to \eqref{e32} and \lemref{l5}), we obtain $f(x)\le J(x,R)$, i.e. $x\in I_{R}\cup C_{R}$. This, together with $x\notin R$, implies $x\notin \Theta(R)$. We thus obtain $(R)^c\subseteq \Theta(R)^c$,  and can now conclude that $R = \Theta(R)$, i.e. $R\in\cE$. Finally, for each $x\in\X$, since $J(x,T_n)$ is non-decreasing in $n$, we have
\[
J(x,R) = \lim_{n\to\infty} J(x,T_n) \ge J(x,T_n)\ge J(x,R_n) \quad \forall n\in\N,
\] 
where the last inequality follows from \propref{p2}. 
\end{proof}


\subsection{The Main Result}

\begin{theorem}\label{t1}
Suppose Assumption \ref{a1}, \eqref{e34}, and \eqref{e32} hold. Then, 
\begin{equation}\label{R*}
R^* := \bigcap_{R\in\cE,\ R\ \text{closed}} R.
\end{equation}
is an optimal equilibrium.
\end{theorem}

\begin{proof}
By construction, $R^*$ is closed and thus belongs to $\B(\X)$. Set $\cE':= \{R\in\cE : R\ \hbox{is closed}\}$. Since the indicator function $1_R$ is upper semicontinuous on $\X$ for all $R\in\cE'$, Proposition 4.1 in \cite{BS12} asserts the existence of a countable subset $(R_n)_{n\in\mathbb{N}}$ of $\cE'$ such that
\begin{equation}\label{e8}
1_{R^*}=\inf_{R\in\cE'}1_R=\inf_{n\in\N}1_{R_n}.
\end{equation}
It follows that $R^*=\bigcap_{n\in\N}R_n$. By Proposition~\ref{p2'}, $R^*$ is an equilibrium, and it remains to show that it is optimal as well. For any $R\in\cE$, recall from Lemma~\ref{lem:bar R in E} (ii) that $\overline R\in \cE$. The definition of $R^*$ then implies $R^*\subseteq \overline R$.  It follows that 
\[
V(x,R^*)= f(x)\vee J(x,R^*) \ge f(x)\vee J(x,\overline R) =   f(x)\vee J(x, R) = V(x,R)\quad\forall\,x\in\X,
\]
where the inequality follows from Lemma~\ref{t3} and the second equality is due to Remark~\ref{rem:value unchanged}. This readily shows that $R^*$ is an optimal equilibrium.
\end{proof}

\begin{remark}
Condition \eqref{e32} is necessary for Theorem~\ref{t1}, and more generally for the existence of an optimal equilibrium. In Section~\ref{sec:examples}, we will show that if \eqref{e32} is violated, an optimal equilibrium may fail to exist; see particularly Proposition~\ref{prop:nu in between} and Remark~\ref{rem:no optimal E}. 
\end{remark}


\section{Characterization of Closed Optimal Equilibria}\label{sec:uniqueness}

Theorem~\ref{t1} shows that there exists an optimal equilibrium $R^*$, which is by construction closed. Whether it is the {\it unique} optimal equilibrium has not yet been addressed. In fact, if we do not impose any closedness condition on optimal equilibria, it is easy to see that uniqueness fails. Indeed, under Assumption \ref{a1}, any subset $R$ of $R^*$ with $\overline{R}= R^*$ is again an equilibrium (Lemma~\ref{lem:bar R in E} (ii)), with $V(x,R)= V(x,R^*)$ for all $x\in\X$ (Remark~\ref{rem:value unchanged}). Thus, $R$ is also an optimal equilibrium. 

This section investigates whether $R^*$ in \eqref{R*} is the unique {\it closed} optimal equilibrium. The eventual conclusion is that $R^*$ in general is not the unique one, but other closed optimal equilibria can only differ from $R^*$ in very limited ways. In addition, we propose a useful sufficient condition, satisfied by many practical stopping problems, for $R^*$ to be the unique closed optimal equilibrium. 


As the first step, we study the relation between an equilibrium and an optimal equilibrium. 

\begin{lemma}\label{lem:T* contains I_R*}
Let $T^*\in\cE$ be an optimal equilibrium.  For any $R\in\cE$,  we have $T^*\setminus R \subseteq I_{R}$.   
\end{lemma}

\begin{proof}
Fix $x\in {T^*}\setminus R$. Since ${T^*}$ is an equilibrium, $x\in {T^*}$ implies $f(x)\ge J(x,{T^*})$, and thus $V(x,{T^*})=f(x)$. Similarly, since ${R}$ is an equilibrium, $x\notin {R}$ implies $f(x) \le J(x,{R})$, and thus $V(x,R)=J(x,R)$. It follows that
\begin{equation}\label{V=V}
V(x,T^*)=f(x) \le J(x,R)=V(x,R).
\end{equation}
With $T^*$ being an optimal equilibrium, the above relation must hold with equality, which implies $x\in I_{R}$. We therefore conclude $T^*\setminus R \subseteq I_{R}$. 
\end{proof}

When $R^*$ in \eqref{R*} belongs to $\cE$, by the definition of an equilibrium (Definition~\ref{def:E}), we have
\begin{equation*}
(R^*)^c\subseteq C_{R^*} \cup I_{R^*}.
\end{equation*}
If it happens that the indifference region $I_{R^*}$ does not exist outside of $R^*$, i.e.
\begin{equation}\label{for uniqueness}
(R^*)^c=C_{R^*},
\end{equation}
then Lemma~\ref{lem:T* contains I_R*} already implies the uniqueness of closed optimal equilibria.  

\begin{theorem}\label{t2}
If $R^*$ in \eqref{R*} is an optimal equilibrium satisfying \eqref{for uniqueness}, then $R^*$ is the unique closed optimal equilibrium. 
\end{theorem}

\begin{proof}
By the closedness of $R^*$, \eqref{for uniqueness} implies $I_{R^*}\subseteq \overline{R^*} = R^*$. 
Let $T^*\in\cE$ be a closed optimal equilibrium. In view of \eqref{R*}, $R^*\subseteq T^*$. On the other hand, Lemma~\ref{lem:T* contains I_R*} implies $T^*\setminus R^* \subseteq I_{R^*}\subseteq R^*$, which entails $T^*\setminus R^*=\emptyset$. We therefore conclude that $T^*=R^*$.
\end{proof}

Theorems~\ref{t1} and \ref{t2} together yield the useful result.

\begin{corollary}
Suppose Assumption \ref{a1}, \eqref{e34}, and \eqref{e32} hold. If $R^*$ in \eqref{R*} satisfies \eqref{for uniqueness}, then $R^*$ is the unique closed optimal equilibrium.
\end{corollary}

\begin{remark}
The condition \eqref{for uniqueness} for uniqueness is satisfied by many practical stopping problems, including those studied in Section~\ref{subsec:Bessel} and Section~\ref{subsec:put on GBM}.  
\end{remark}

In general, $R^*$ in \eqref{R*} may not always satisfy \eqref{for uniqueness}, and there can be multiple closed optimal equilibria. This is demonstrated in the next example.

\begin{example}\label{eg:counterexample}
Let $X$ be a one-dimensional Bessel process, i.e. $X_t := |W_t|$, where $W$ is a one-dimensional Brownian motion. Consider the hyperbolic discount function in \eqref{hyperbolic}. Let $a^*>0$ be specified as in Lemma~\ref{lem:previous result} below, and pick $b^*>a^*$. Define the payoff function $f:[0,\infty)\to\R_+$ by  
\begin{equation*}
f(x) = 
\begin{cases}
x\quad &\hbox{if}\ 0\le x\le b^*,\\
\E^x\left[ \delta(T^x_{b^*}) b^* \right]=\E^x\left[\frac{b^*}{1+\beta T^x_{b^*}}\right]\quad &\hbox{if}\ x> b^*,
\end{cases}
\end{equation*}  
where $T^x_{b^*}$ is defined as in \eqref{T^x_y}. Note specifically that $f$ is continuous, with $f(b^*)=b^*$. Under current setting, one may conclude from Theorem~\ref{t1} that $R^*$ in \eqref{R*} is an optimal equilibrium. 

Observe that for any $x\in[0,a^*)$, $J(x,[a^*,b^*]) = J(x,[a^*,\infty)) > x = f(x)$, where the inequality follows from Lemma~\ref{lem:previous result}; for any $x> b^*$, $J(x,[a^*,b^*]) = \E^x\left[\frac{b^*}{1+\beta T^x_{b^*}}\right]=f(x)$. This implies $[0,a^*)\cup (b^*,\infty)\in I_{[a^*,b^*]}\cup C_{[a^*,b^*]}$, and thus $[a^*,b^*]\in \cE$.  

We claim that $R^* = [a^*,b^*]$. By \eqref{R*}, $R^*\subseteq [a^*,b^*]$. If $R^*$ is not a connected set, one may argue as in \cite[Lemma 4.3]{HN17} to reach a contradiction. Thus, $R^*= [a,b]$ for some $a^*\le a\le b\le b^*$. If $a^*<a$, then for any $x\in[0,a)$, $J(x,[a,\infty))= J(x,R^*) \ge x$, where the inequality follows from $R^*\in\cE$. This implies $[a,\infty)\in\cE$, which contradicts Lemma~\ref{lem:previous result} as $a> a^*$. If $b< b^*$, then
\[
x> \E^x\left[\frac{b}{1+\beta T^x_{b}}\right] = J(x,R^*),\quad \forall x\in (b,b^*). 
\]  
This contradicts $R^*\in \cE$. We therefore conclude that $R^*=[a^*,b^*]$. 

Now, note that $(b^*,\infty)\subset I_{R^*}$, and thus \eqref{for uniqueness} fails to hold. Moreover, $T^*:=[a^*,\infty)$ is also a closed optimal equilibrium, as 
\[
V(x,T^*) = f(x) = \E^x\bigg[\frac{b^*}{1+\beta T^x_{b^*}}\bigg] = \E^x\bigg[\frac{f(b^*)}{1+\beta T^x_{b^*}}\bigg] = J(x,R^*) = V(x,R^*),\quad \forall x> b^*. 
\]
\end{example}

In view of Example~\ref{eg:counterexample}, it is of interest to investigate, when \eqref{for uniqueness} is not guaranteed, how an arbitrary closed optimal equilibrium $T^*$ can differ from $R^*$ in \eqref{R*}. 

To this end, take $R\in\cE$ that is closed. For any $x\notin R$, define 
\begin{equation}\label{ell}
\ell_R(x) := \sup\{y\in R: y<x\}\quad \hbox{and}\quad r_R(x):= \inf\{y\in R : y>x\}.
\end{equation}
If there is no $y\in R$ with $y<x$, we take $\ell_R(x) = \inf\X$ (which can be $-\infty$). Similarly, if there is no $y\in R$ with $y>x$, we take $r_R(x) = \sup \X$ (which can be $\infty$).
Note that  $\ell_R(x)<x$ and $r_R(x)>x$, thanks to the closedness of $R$. Let $T^*\in\cE$ be a closed optimal equilibrium that contains $R$. Then $T^*\setminus R$ must intersect $(\ell_R(x), r_R(x))$ for some $x\notin R$. To state a precise relation between $T^*$ and $R$, we will require the state process $X$ to be {\it regular} in the following sense: 
\begin{equation}\label{regular}
\hbox{for any $x\in \text{int}(\X)$,}\quad \P^x(T^x_y <\infty) >0\quad \forall y\in \X,
\end{equation}
where 
\begin{equation}\label{T^x_y}
T^x_y := \inf\{t\ge 0 : X^x_t =y\}
\end{equation}
is the first hitting time of $X$ to $y$ from $x$, for any distinct $x,y\in\X$. 
This is in line with the formulation in Karlin and Taylor \cite[Chapter 15]{KT-book-81}.

\begin{proposition}\label{prop:T* and R}
Suppose that 
\eqref{e34} holds, $X$ is regular in the sense of \eqref{regular}, and \eqref{e1} holds with strict inequality. Let $T^*\in\cE$ be a closed optimal equilibrium, and $R\in\cE$ be closed and contained in $T^*$. Then, for any $x\notin R$, 
\begin{itemize}
\item [(i)] if $(\ell_R(x), r_R(x))\cap C_{R}\neq\emptyset$, then $T^*\cap (\ell_R(x), r_R(x)) =\emptyset$;
\item [(ii)] if $(\ell_R(x), r_R(x))\cap C_{R}=\emptyset$ and $f\not\equiv 0$ on $(\ell_R(x),r_R(x))$, then 
\[
T^*\cap (\ell_R(x), r_R(x)) =\emptyset\quad \hbox{or}\quad (\ell_R(x), r_R(x))\subseteq T^*;
\]
\item [(iii)] if $(\ell_R(x), r_R(x))\cap C_{R}=\emptyset$ and $f\equiv 0$ on $(\ell_R(x),r_R(x))$, then for any Borel subset $A$ of $(\ell_R(x), r_R(x))$, $R\cup A\in\cE$ and $V(y,R\cup A) = V(y,T^*)=0$ for all $y\in (\ell_R(x), r_R(x))$.
\end{itemize}  
\end{proposition} 

\begin{proof}
(i) 
By contradiction, suppose $T^*\cap (\ell_R(x), r_R(x)) \neq \emptyset$. Take $z\in (\ell_R(x), r_R(x))\cap C_{R}$. By Lemma~\ref{lem:T* contains I_R*}, $z\notin T^*$, and we can thus define
\begin{equation}\label{ell'}
\ell(z) := \sup\{y\in T^*: y<z\} \quad \hbox{and}\quad r(z):= \inf\{y\in T^* : y>z\}.
\end{equation}
We must have either $\ell(z) > \ell_R(x)$ or $r(z) < r_R(x)$, otherwise $T^*\cap (\ell_R(x), r_R(x))\neq\emptyset$ would be violated.
Also, by the closedness of $T^*$, we have $\ell(z)<z$ and $r(z)>z$. 
Note that $(\ell_R(x), r_R(x))\cap C_{R}\neq\emptyset$ already excludes the following cases: 
\begin{itemize}
\item [(a)] $\ell_R(x)=\inf \X\notin \X$, and $r_R(x)=\sup\X\notin \X$; 
\item [(b)] $\ell_R(x)=\inf \X\in \X$ (or $\ell_R(x)>\inf\X$), $r_R(x)=\sup\X\notin\X$, and $f(\ell_R(x))=0$;
\item [(c)] $\ell_R(x)=\inf\X\notin\X$, $r_R(x)=\sup\X\in\X$ (or $r_R(x)<\sup\X$), and $f(r_R(x))=0$;
\item [(d)] $\ell_R(x)=\inf \X\in \X$ (or $\ell_R(x)>\inf\X$), $r_R(x)=\sup\X\in\X$ (or $r_R(x)<\sup\X$), and $f(\ell_R(x))=f(r_R(x))=0$. 
\end{itemize}
Indeed, in each of these cases, $J(y,R) = 0 \le f(y)$ for all $y\in (\ell_R(x),r_R(x))$, thanks to \eqref{e34}. This implies $(\ell_R(x), r_R(x))\cap C_{R}=\emptyset$, a contradiction. 

Let us focus on the remaining cases: 
\begin{itemize}
\item [(e)] $\ell_R(x)=\inf \X\in \X$ (or $\ell_R(x)>\inf\X$), $r_R(x)=\sup\X\notin \X$, and $f(\ell_R(x))>0$;
\item [(f)] $\ell_R(x)=\inf \X\notin \X$, $r_R(x)=\sup\X\in\X$ (or $r_R(x)<\sup\X$), and $f(r_R(x))>0$; 
\item [(g)] $\ell_R(x)=\inf \X\in \X$ (or $\ell_R(x)>\inf\X$), $r_R(x)=\sup\X\in\X$ (or $r_R(x)<\sup\X$), and either $f(\ell_R(x))>0$ or $f(r_R(x))>0$ (Assume without loss of generality that $f(\ell_R(x))>0$).
\end{itemize}
Define $\tau := \rho(z,T^*)$ and $v := \rho(z,R)$. By definition, $\tau\le v$. We consider the set 
$$A:= \{\omega\in\Omega:\tau < v\},$$ 
and claim that $\P^z(A) >0$. If $\P^z(A) = 0$, we have either ``$\ell(z)=\ell_R(x)$ and $r(z)< r_R(x)$'' or ``$\ell(z)>\ell_R(x)$ and $r(z)= r_R(x)$''. Assume the former, without loss of generality. Note that $\P^z(A) = 0$ entails $T^z_{\ell_R(x)}<T^z_{r(z)}$ $\P^z$-a.s. This, together with $X$ being a strong Markov process, implies that $X$, when starting from $z$, can never reach the region $[r(z),\infty)$, $\P^z$-a.s. This, however, contradicts $X$ being regular in the sense of \eqref{regular}. With $\P^z(A)>0$, $X$ being a regular process then implies $\P^z\big(X^z_{v}= \ell_R(x)\ |\ A\big) >0$ in cases (e) and (g), and $\P^z\big(X^z_{v}= r_R(x)\ |\ A\big) >0 $ in case (f). We therefore conclude that in all these three cases,
\begin{equation}\label{P>0}
\P^z\left(f\left(X_{v}\right) >0\ \middle|\ A\right) >0. 
\end{equation}
Now, we carry out the same calculation in \eqref{ee1}, with $x$ and $T$ replaced by $z$ and $T^*$. In this calculation, note that we now have a {\it strict} inequality
\[
\E^z\left[\E^z\left[\delta(v)f(X_v)\mid\F_\tau\right] 1_A \right] > \E^z \left[\delta(\tau)\E^z\left[\delta(v-\tau)f(X_v) \mid\F_\tau\right] 1_A \right].
\]
We get ``$>$'', instead of merely ``$\ge$'', because of $\P(A)>0$,  \eqref{P>0}, $0<\tau< v$ on $A$, and \eqref{e1} with strict inequality. 
Then, the same argument in Lemma~\ref{t3}, below \eqref{ee1}, can be applied here, which yields the strict inequality $J(z,R)>J(z,T^*)$. Since $z\notin T^*$ and $T^*$ is an equilibrium, $J(z,T^*)\ge f(z)$. Thus, we get $V(z,R)\ge J(z,R) > J(z,T^*) = V(z,T^*)$, which contradicts the fact that $T^*$ is an optimal equilibrium. 

(ii) Since $R$ is an equilibrium, $(\ell_R(x), r_R(x))\cap C_{R}=\emptyset$ implies $(\ell_R(x), r_R(x))\subseteq I_{R}$.  By contradiction, suppose $T^*\cap (\ell_R(x), r_R(x)) \neq \emptyset$, but $(\ell_R(x), r_R(x))$ is not contained in $T^*$. Take $z\in (\ell_R(x), r_R(x))\setminus T^*$, and define $\ell(z)$ and $r(z)$ as in \eqref{ell'}. Note that cases (a)-(d) specified above are again excluded under current setting. Indeed, in these cases  $J(y,R) = 0$ for all $y\in (\ell_R(x),r_R(x))$. With $(\ell_R(x), r_R(x))\subseteq I_{R}$, we conclude $f(y) = J(y,R) =0$ for all $y\in (\ell_R(x),r_R(x))$, which contradicts the choice of $f$. For the remaining cases (e)-(g) specified above, the same argument as in (i) shows that $V(z,R) > V(z,T^*)$, which contradicts the fact that $T^*$ is an optimal equilibrium. 

(iii) For any Borel subset $A$ of $(\ell_R(x),r_R(x))$, thanks to $f\equiv 0$ on $(\ell_R(x),r_R(x))$ and \eqref{e34}, $J(y,R\cup A) = 0= f(y)$ for all $y\in (\ell_R(x),r_R(x))$. That is, $(\ell_R(x),r_R(x))\subseteq I_{R\cup A}$, which already implies $R\cup A\in\cE$. Recall from (ii) that $R$ being an equilibrium and $(\ell_R(x), r_R(x))\cap C_{R}=\emptyset$ imply $(\ell_R(x), r_R(x))\subseteq I_{R}$. Thus, we have $V(y,R\cup A) = f(y) =0$ for all $y\in (\ell_R(x), r_R(x))$. In particular, if we take $A^* := T^* \cap (\ell_R(x), r_R(x))$, then for all $y\in (\ell_R(x), r_R(x))$, $V(y,T^*) = V(y,R\cup A^*)=0$.
\end{proof}

\begin{remark}
The condition ``\eqref{e1} holds with strict inequality'' for Proposition~\ref{prop:T* and R} is not restrictive in practice. Most commonly-seen non-exponential discount functions, including those specified under \eqref{e1}, satisfy this condition.
\end{remark}

Recall $R^*$ in \eqref{R*}. For each $x\notin R^*$, we define $\ell_*(x)$ and $r_*(x)$ as in \eqref{ell}, with $R$ replaced by $R^*$. Any closed optimal equilibrium $T^*$ can differ from $R^*$ in very limit ways, as stated below. 

\begin{corollary}\label{coro:T* and R*}
Suppose that \eqref{e34} holds, $X$ is regular in the sense of \eqref{regular}, and \eqref{e1} holds with strict inequality. Let $T^*\in\cE$ be a closed optimal equilibrium. If $R^*$ in \eqref{R*} is an optimal equilibrium, then for any $x\notin R^*$, 
\begin{itemize}
\item [(i)] if $(\ell_*(x), r_*(x))\cap C_{R^*}\neq\emptyset$, then $T^*\cap (\ell_*(x), r_*(x)) =\emptyset$;
\item [(ii)] if $(\ell_*(x), r_*(x))\cap C_{R^*}=\emptyset$ and $f\not\equiv 0$ on $(\ell_*(x),r_*(x))$, then 
\[
T^*\cap (\ell_*(x), r_*(x)) =\emptyset\quad \hbox{or}\quad (\ell_*(x), r_*(x))\subseteq T^*;
\]
\item [(iii)] if $(\ell_*(x), r_*(x))\cap C_{R^*}=\emptyset$ and $f\equiv 0$ on $(\ell_*(x),r_*(x))$, then for any Borel subset $A$ of $(\ell_*(x), r_*(x))$, $R^*\cup A$ is an optimal equilibrium. 
\end{itemize}  
\end{corollary}

\begin{proof}
By definition, $R^*$ is closed and $R^*\subseteq T^*$. The result then follows directly from Proposition~\ref{prop:T* and R} and $R^*$ being an optimal equilibrium. 
\end{proof}

\begin{remark}
One may wonder if Corollary~\ref{coro:T* and R*} (ii) can be strengthened so that the possibility $(\ell_*(x), r_*(x))\subseteq T^*$ can be excluded. If this could be done, Corollary~\ref{coro:T* and R*} (i) and (ii) together would imply that $R^*$ is the unique closed equilibrium, whenever $f\not\equiv 0$ on any sub-interval of $\X$. 

This, however, cannot be done in general. Indeed, in Example~\ref{eg:counterexample}, $T^*\setminus R^* = (b^*,\infty) \subseteq I_{R^*}$, which shows that excluding ``$(\ell_*(x), r_*(x))\subseteq T^*$'' in Corollary~\ref{coro:T* and R*} (ii) is not possible in general.
\end{remark}







\subsection{Comparison with the Discrete-Time Setting in \cite{HZ17-discrete}}\label{subsec:comparison}
In a discrete-time setting, Huang and Zhou \cite{HZ17-discrete} also establish the existence of an optimal equilibrium, for an infinite-horizon stopping problem under non-exponential discounting. There is, nonetheless, a key difference between the present paper and \cite{HZ17-discrete}. When defining the fixed-point operator $\Theta$ in \eqref{Theta}, we consider the indifference region, consistent with the iterative approach originally proposed in \cite{HN17}. By contrast, the indifference region does not play a role in \cite{HZ17-discrete}: it is simply taken as a part of the stopping region. This simplification allows \cite{HZ17-discrete} to establish strong, desirable results in discrete time, including (i) every equilibrium is closed, and (ii) there exists a {\it unique} optimal equilibrium, under mild continuity assumptions.

While it is tempting to make the same simplification in continuous time, such a simplification is legitimate {\it only} in discrete time. Specifically, if we ignore the indifference region in \eqref{regions} by including it in the stopping region, it will be questionable whether fixed-point iterations can converge to equilibria (i.e. Proposition~\ref{prop:iteration} may no longer hold). This goes back to the proof of the convergence result \cite[Theorem 3.16]{HN17}, which requires the use of the indifference region. In the discrete-time setting of \cite{HZ17-discrete}, the absence of the indifference region is salvaged by \cite[Lemma 3.1]{HZ17-discrete}, which allows a fixed-point iteration to converge to an equilibrium (\cite[Theorem 3.1]{HZ17-discrete}), without the need of the indifference region. Such a result, however, requires $\Theta(R)\subseteq R$ for $R\in \B(\X)$. In continuous time, this relation does not hold (see Remark~\ref{rem:R subset Theta R}), unless $R$ is already an equilibrium.


\section{Examples}\label{sec:examples}

To illustrate our theoretic results, we study three distinct stopping problems in this section. The first one is the stopping of a Bessel process, investigated in detail in \cite[Section 4]{HN17}. We will briefly present it, as it perfectly confirms Theorem~\ref{t1}. The second example is the stopping of a geometric Brownian motion, while the third is an American put option written on a geometric Brownian motion. These two examples demand a careful comprehensive analysis, which will enhance our understanding, particularly, of the role of the integrability condition \eqref{e32}.
In each of these three examples, we will characterize an optimal equilibrium through an explicit formula. 



\subsection{Stopping of a Bessel Process}\label{subsec:Bessel}
Consider a one-dimensional Bessel process $X$, given by $X_t := |W_t|$, where $W$ is a one-dimensional Brownian motion. The state space is then $\X=[0,\infty)$. Let the payoff function be $f(x)=x$ on $\X$, and the discount function be of the hyperbolic type as in \eqref{hyperbolic}. The expected payoff $J(x,R)$ in \eqref{J} then takes the form \eqref{J GBM}. Also, from standard properties of the Brownian motion $W$, one can verify that Assumption~\ref{a1}, \eqref{e34}, and \eqref{e32} are all satisfied. 

As detailed in \cite[p.82]{HN17}, the current setting relates to a real options problem in which the management of a large non-profitable insurance company intends to liquidate or sell the company, and would like to decide when to do so.

The next result presents a precise characterization of all closed equilibria, which is a direct consequence of \cite[Proposition 4.5 and Lemma 4.4]{HN17}.

\begin{lemma}\label{lem:previous result}
$R\in\B(\X)$ is a closed equilibrium if and only if $R = [a,\infty)$ for some $a\in [0,a^*]$, where $a^*>0$ is characterized by $a^*\int_0^\infty e^{-s}\sqrt{2\beta s}\tanh(a^*\sqrt{2\beta s})ds =1$. Moreover, for any $a\in (0,a^*]$, $J(x,[a,\infty)) > x$ for all $x\in[0,a)$. 
\end{lemma}

\begin{proposition}\label{prop:a*}
$R^* = [a^*,\infty)$, with $a^*>0$ as in Lemma~\ref{lem:previous result}, is the unique closed optimal equilibrium. 
\end{proposition}

\begin{proof}
In view of \eqref{R*} and Lemma~\ref{lem:previous result}, $R^* = \bigcap_{a\in[0,a^*]} [a,\infty) = [a^*,\infty)$. By Theorem~\ref{t1}, $R^*= [a^*,\infty)$ is an optimal equilibrium. Thanks to Lemma~\ref{lem:previous result} again, $J(x,[a^*,\infty)) > x$ for all $x\in[0,a^*)$, which implies $(R^*)^c = [0,a^*)= C_{R^*}$. Theorem~\ref{t2} then asserts that $R^*$ is the unique closed optimal equilibrium. 
\end{proof}

Since the payoff function $f(x) =x$ is strictly increasing, it is reasonable to focus on threshold-type stopping policies $[a,\infty)$, as indicated in Lemma~\ref{lem:previous result}. The challenge here is how large $a\ge 0$ should be, so that the resulting policy can be an equilibrium. 
Intuitively, if $a\ge 0$ is too large, the discounting involved $\frac{1}{1+\beta \rho(x,[a,\infty))}$ may outweigh the payoff $a$ obtained at stopping, reducing the expected payoff $J(x,[a,\infty)) = \E^x\big[\frac{a}{1+\beta \rho(x,[a,\infty))}\big]$. Thus, anticipating that future selves will follow too large a threshold $a\ge 0$, the current self may decide to stop immediately -- such that $[a,\infty)$ cannot be an equilibrium. The main contribution of Lemma~\ref{lem:previous result} is to provide a sharp upper bound $a^*$ for the threshold $a\ge 0$ such that $[a,\infty)$ is an equilibrium. 
Now, since the discount function \eqref{hyperbolic} induces decreasing impatience, the agent prefers smaller equilibria over larger ones; see Corollary~\ref{coro:smaller better} and the detailed discussion below it. This readily shows, at least at the intuitive level, that the smallest possible equilibrium $[a^*,\infty)$ is optimal, which is rigorously established in Proposition~\ref{prop:a*}.

That $[a^*,\infty)$ is an optimal equilibrium can in fact be derived from \cite[Proposition 4.10]{HN17}, which was based on a direct calculation specifically for the current stopping problem. Here, we see that our general theoretic results recover the same conclusion.


\subsection{Stopping of a Geometric Brownian Motion}\label{subsec:GBM}
Consider a geometric Brownian motion $X$ given by \eqref{GBM}, with state space $\X=(0,\infty)$, as well as the payoff function $f(x)=x$ on $\X$ and the hyperbolic discount function specified in \eqref{hyperbolic}. The expected payoff $J(x,R)$ in \eqref{J} then takes the form \eqref{J GBM}. The current setting describes the problem where an agent, who discounts hyperbolically, wants to find a suitable time to liquidate the asset $X$. 

Note that $X$ satisfies Assumption~\ref{a1}, thanks to Remark~\ref{rem:Assumption satisfied}. The conditions \eqref{e34} and \eqref{e32}, however, do {\it not} always hold.
Recall the constant $\nu\in\R$ defined in \eqref{nu}. 

\begin{remark}\label{rem:conditions not satisfied}
In view of \eqref{X(infty)}, \eqref{e34} is satisfied only when $\nu \le 0$. On the other hand, \eqref{e32} is violated whenever $\nu > -1/2$. Indeed, since $\nu> -1/2$ is equivalent to $\mu >0$, for any $x>0$, 
\begin{align*}
\E^x\bigg[\sup_{t\in[0,\infty)}\delta(t) X_t\bigg] &=\lim_{T\to\infty} \E^x\bigg[\sup_{t\in[0,T]}\delta(t) X_t\bigg]  \ge \lim_{T\to\infty} \E^x\left[\delta(T) X_T\right] = \lim_{T\to\infty} \frac{e^{\mu T}}{1+\beta T}= \infty.
\end{align*}
\end{remark}

The fact that \eqref{e34} and \eqref{e32} need not always hold poses a challenge to finding optimal equilibria. First, $R^*$ in \eqref{R*} is no longer a guaranteed optimal equilibrium, as Theorem~\ref{t1} may not be applicable. Even the fundamental result that the intersection of equilibria remains an equilibrium is now in question, as Propositions~\ref{p2} and \ref{p2'} require \eqref{e34} and \eqref{e32} both. In the following, we will deal with the cases $\nu>0$, $\nu\in (-1/2,0]$, and $\nu\le -1/2$ separately. We will see that results in Section~\ref{sec:existence} can still be applied, after we modify our current stopping problem appropriately. 

For $\nu>0$, an optimal equilibrium readily exists, thanks to Example~\ref{eg:nu>0}. 

\begin{proposition}\label{prop:nu>0}
Let $\nu>0$. Then $R^* = \emptyset$ is an optimal equilibrium. 
\end{proposition}

\begin{proof}
From Example~\ref{eg:nu>0}, we already know that $\emptyset$ is an optimal equilibrium. Then $R^* = \emptyset$, simply from its definition in \eqref{R*}. 
\end{proof}

For $\nu \le -1/2$, we first observe that any closed equilibrium has to take a specific form. 

\begin{lemma}\label{lem:form of cE}
Let $\nu\le -1/2$. For any $R \in\cE$ that is closed, $R = (0,\infty)$ or $R= [a,\infty)$ for some $a>0$.  
\end{lemma}

\begin{proof}
For any $R\in\cE$ that is closed, observe that $R\neq \emptyset$. Indeed, if $R=\emptyset$, \eqref{X(infty)} implies $J(x,R)= 0 < x$ for all $x>0$, which contradicts $R$ being an equilibrium. Assume $R\neq\emptyset$, and take $a := \inf R\ge 0$. By contradiction, suppose there exists $x\in (a,\infty)$ such that $x\notin R$. Define
\begin{equation}\label{bounds}
\ell := \sup\{y\in R: y<x\}\quad \hbox{and}\quad r:= \inf\{y\in R : y>x\},
\end{equation}
where the infimum is taken to be $\infty$ if there exists no $y\in R$ with $y>x$. 
By the closedness of $R$, we have $\ell <x$ and $r>x$, and thus $\rho(x,R)>0$ $\P^x$-a.s. For $\nu\le -1/2$, or equivalently $\mu\le 0$, $X$ is a supermartingale. It follows that
$
J(x,R) < \E^x\left[X_{\rho(x,R)}\right]  \le x,  
$
where the first inequality follows from $\rho(x,R)>0$ and the second the supermartingale property. This implies $x\in S_R$, which contradicts the fact that $R$ is an equilibrium. 
\end{proof}

To understand the behavior of $J(x,R)$, given in \eqref{J GBM}, we need the next technical lemma.  

\begin{lemma}\label{lem:kappa}
For any $a>0$, define 
\begin{equation}\label{kappa}
\kappa(x,a) := \E^x\left[\frac{a}{1+\beta T^x_a}\right]\quad \hbox{for}\ \ 0< x\le a,
\end{equation}
where $T^x_a$ is defined as in \eqref{T^x_y}. Then, 
\begin{itemize}
\item [(i)] $x\mapsto \kappa(x,a)$ is strictly increasing on $(0,a)$, with $\lim_{x\to 0}\kappa(x,a) =0$ and $\kappa(a,a)=a$. 
\item [(ii)] If $\nu\le -1/2$, $x\mapsto \kappa(x,a)$ is strictly convex on $(0,a)$. 
\item [(iii)] If $\nu >-1/2$, then $\lim_{x\to 0}\kappa_{x}(x,a)=\infty$, and $x\mapsto \kappa(x,a)$ is either strictly concave on $(0,a)$, or strictly concave on $(0,x^*)$ and strictly convex on $(x^*,a)$ for some $x^*\in (0,a)$.
\end{itemize}
\end{lemma}

\begin{proof}
By definition, $\kappa(a,a)=a$. Let $p(t)$ denote the density function of $T^x_a$. Observe that
\begin{align*}
\kappa(x,a) &= a \int_0^\infty \frac{ p(t) }{1+\beta t}dt= a \int_0^\infty \int_0^\infty e^{-(1+\beta t)s} p(t) ds dt\\
&= a \int_0^\infty e^{-s} \E^x\left[e^{-\beta s T^x_{a}}\right] ds= a \int_0^\infty e^{-s} \left(\frac{x}{a}\right)^{\sqrt{\nu^2+\frac{2\beta s}{\sigma^2}}-\nu} ds,
\end{align*}
where the last equality follows from \cite[Formula 2.0.1 on p. 628]{BS-book-2002}. This particularly shows that $\lim_{x\to 0}\kappa(x,a)=0$. Consider 
\begin{equation}\label{g(s,nu)}
g(s,\nu) := \sqrt{\nu^2+2\beta s/\sigma^2}-\nu.
\end{equation}
By definition, $g(s,\nu)>0$ for all $s>0$. Observe that
\begin{equation}\label{kappa_x}
\kappa_x(x,a) = \int_0^\infty e^{-s} g(s,\nu)\left(\frac{x}{a}\right)^{g(s,\nu)-1} ds\ >0\qquad \forall x\in (0,a),
\end{equation}
and thus $x\mapsto \kappa(x,a)$ is strictly increasing on $(0,a)$. 

If $\nu\le -1/2$, then $g(s,\nu) > |\nu|-\nu \ge1$ for all $s>0$. It follows that $x\mapsto \kappa_{x}(x,a)$ is strictly increasing on $(0,a)$, i.e. $x\mapsto \kappa(x,a)$ is strictly convex on $(0,a)$. 

If $\nu > -1/2$, there must exist $s^*>0$ such that $g(s,\nu)-1<0$ if and only if $s<s^*$. Then $\kappa_{x}(x,a)$ can be re-written as 
\begin{align}
\kappa_{x}(x,a) = &\int_0^{s^*} e^{-s} g(s,\nu)\left(\frac{a}{x}\right)^{1-g(s,\nu)} ds+ \int_{s^*}^\infty e^{-s} g(s,\nu) \left(\frac{x}{a}\right)^{g(s,\nu)-1} ds.\label{kappa_x'} 
\end{align}
Observe that $\lim_{x\to 0}\kappa_{x}(x,a)=\infty$, as the first term on the right hand side above explodes, while the second term vanishes. 
Moreover, from \eqref{kappa_x},
\begin{align}\label{kappa_xx}
\kappa_{xx}(x,a) &= \frac{1}{a} \int_0^\infty e^{-s} g(s,\nu) \big( g(s,\nu)-1 \big)\left(\frac{x}{a}\right)^{g(s,\nu)-2} ds = \frac{a}{x^2} q(x),
\end{align}
where
\begin{align*}
q(x)&:= \int_0^{s^*} e^{-s} g(s,\nu) \big(g(s,\nu)-1 \big)\left(\frac{x}{a}\right)^{g(s,\nu)} ds+ \int_{s^*}^\infty e^{-s} g(s,\nu) \big(g(s,\nu)-1 \big)\left(\frac{x}{a}\right)^{g(s,\nu)} ds.
\end{align*}
We claim that $q$ is strictly convex on $(0,a)$. Indeed, a direct calculation yields
\begin{align*}
q'(x)&:= \frac{1}{a}\bigg[\int_0^{s^*} e^{-s} g^2(s,\nu) \big(g(s,\nu)-1 \big)\left(\frac{a}{x}\right)^{1-g(s,\nu)} ds\\
&\hspace{0.6in}+ \int_{s^*}^\infty e^{-s} g^2(s,\nu) \big(g(s,\nu)-1 \big)\left(\frac{x}{a}\right)^{g(s,\nu)-1} ds\bigg].
\end{align*}
Since $g(s,\nu)-1<0$ if and only if $s<s^*$, each of the two terms on the right hand side above is strictly increasing in $x$; specifically, the first term becomes less negative, while the second term becomes more positive, as $x$ increases. Thus, $q$ is strictly convex on $(0,a)$. By definition, $q(0)=0$. The strict convexity of $q$ then entails  either of the two cases: (i) $q(x)>0$ on $(0,a)$, or (ii) $q(x)<0$ on $(0,x^*)$ and $q(x)>0$ on $(x^*,a)$, with 
\[
x^* := \inf\{x>0: q(x) = 0\}\wedge a.
\] 
If Case (i) holds, then $x\mapsto\kappa(x,a)$ is strictly convex on $(0,a)$, in view of \eqref{kappa_xx}. This, however, contradicts $\lim_{x\to 0}\kappa_{x}(x,a)=\infty$. Thus, we must be in Case (ii). Thanks to \eqref{kappa_xx} again, if $x^* = a$, then $x\mapsto\kappa(x,a)$ is strictly concave on $(0,a)$; if $x^*< a$, then $x\mapsto\kappa(x,a)$ is strictly concave on $(0,x^*)$, and strictly convex on $(x^*,a)$.  
\end{proof}

\begin{proposition}\label{prop:nu<=-1/2}
Let $\nu\le -1/2$. Then $(0,\infty)$ is the only closed equilibrium. Hence, $R^*=(0,\infty)$ is the unique closed optimal equilibrium. 
\end{proposition}

\begin{proof}
For any $a>0$, by Lemma~\ref{lem:kappa} (i) and (ii), $x> \kappa(x,a) = J(x,[a,\infty))$ for all $x\in (0,a)$, and thus $(0,a)\subseteq S_{[a,\infty)}$. This implies $[a,\infty)\notin \cE$, for all $a>0$. In view of Lemma~\ref{lem:form of cE}, the only closed equilibrium is $(0,\infty)$. Then $R^* = (0,\infty)$, simply by the definition in \eqref{R*}. Since $X$ satisfies Assumption~\ref{a1}, we deduce from Lemma~\ref{lem:bar R in E} (ii) that any $R\in\cE$ must satisfy $\overline R= (0,\infty)$. This, together with Remark~\ref{rem:value unchanged}, yields $V(x,R)=V(x,(0,\infty))$ for all $x>0$, for any $R\in\cE$. Thus, $R^* = (0,\infty)$ is the unique closed optimal equilibrium.
\end{proof}

For $\nu\in (-1/2,0]$, a clear-cut characterization for closed equilibria as in Lemma~\ref{lem:form of cE} cannot be obtained. Instead, we derive the following relation between different forms of closed equilibria, which will prove instrumental in finding optimal equilibria. 

\begin{lemma}\label{lem:two-part implies one-part}
Let $\nu\in (-1/2,0]$. If there exist $0<\ell<r<\infty$ such that $(0,\ell]\cup [r,\infty)\in \cE$, then $[r,\infty)\in\cE$. 
\end{lemma}

\begin{proof}
For any $0<\ell<r<\infty$ such that $(0,\ell]\cup [r,\infty)\in \cE$, we must have
\begin{equation}\label{J>y}
J(y,R) \ge y,\quad \hbox{for all $y\in (\ell,r)$}.
\end{equation}
 For any $y\in(\ell,r)$, define $\tau:=\inf\{t\ge 0: X^y_t\notin (\ell,r)\}$, and let $p:(0,\infty)\times\{\ell,r\}\to [0,\infty)$ denote the joint density function of $(\tau, X^y_{\tau})$. Then, 
\begin{align}\label{Laplace}
J(y,R) &= \E^y\left[\frac{X_{\tau}}{1+\beta \tau}\right]= \ell \int_0^\infty \frac{p(t,\ell)}{1+\beta t} dt +r \int_0^\infty \frac{p(t,r)}{1+\beta t} dt \notag\\
&=\ell \int_0^\infty \int_0^\infty  e^{-(1+\beta t)s}p(t,\ell) ds dt + r \int_0^\infty \int_0^\infty e^{-(1+\beta t)s} p(t,r) ds dt\notag\\
&= \ell \int_0^\infty e^{-s}  \E^y\left[ e^{-\beta s \tau} 1_{\{X^y_{\tau} =\ell\}} \right] ds + r \int_0^\infty e^{-s} \E^y\left[ e^{-\beta s \tau} 1_{\{X^y_{\tau} = r \}} \right]  ds.
\end{align}
Thanks to the formulas of $\E^y\big[ e^{-\beta s \tau} 1_{\{X^y_{\tau} =\ell\}} \big]$ and $\E^y\big[ e^{-\beta s \tau} 1_{\{X^y_{\tau} = r\}} \big]$ in \cite[Formula 3.0.5 on p. 633]{BS-book-2002} and $\nu \le 0$, the above equation becomes
\begin{align}\label{J(y,R)}
J(y,R) &= \ell \int_0^\infty  e^{-s} h_1(s,y,\ell,r)\ ds+ r\int_0^\infty  e^{-s} h_2(s,y,\ell,r)\ ds, 
\end{align}
where
\begin{align*}
h_1(s,y,\ell,r) &: =  \left(\frac{y}{\ell}\right)^{|\nu|} \frac{\left(\frac{r}{y}\right)^{\sqrt{\nu^2+2\beta s/\sigma^2}}-\left(\frac{y}{r}\right)^{\sqrt{\nu^2+2\beta s/\sigma^2}}}{\left(\frac{r}{\ell}\right)^{\sqrt{\nu^2+2\beta s/\sigma^2}}-\left(\frac{\ell}{r}\right)^{\sqrt{\nu^2+2\beta s/\sigma^2}}},\\
h_2(s,y,\ell, r) &:= \left(\frac{y}{r}\right)^{|\nu|} \frac{\left(\frac{y}{\ell}\right)^{\sqrt{\nu^2+2\beta s/\sigma^2}}-\left(\frac{\ell}{y}\right)^{\sqrt{\nu^2+2\beta s/\sigma^2}}}{\left(\frac{r}{\ell}\right)^{\sqrt{\nu^2+2\beta s/\sigma^2}}-\left(\frac{\ell}{r}\right)^{\sqrt{\nu^2+2\beta s/\sigma^2}}}. 
\end{align*}
Note that, by definition, 
\begin{equation}\label{homogeneity}
h_1(s,\alpha y, \alpha \ell, \alpha r) = h_1(s, y,\ell, r)\quad \hbox{and}\quad h_2(s,\alpha y, \alpha \ell, \alpha r) = h_2(s, y,\ell, r),\quad \forall \alpha>0. 
\end{equation}
For any $\alpha>0$, take $R(\alpha) := (0,\alpha\ell]\cup [\alpha r,\infty)$. Observe that, by \eqref{homogeneity}, the same calculation leading to \eqref{J(y,R)} yields
\begin{align*}
J(\alpha y, R(\alpha)) &= \alpha\ell \int_0^\infty  e^{-s} h_1(s,y,\ell,r)\ ds+ \alpha r\int_0^\infty  e^{-s} h_2(s,y,\ell,r)\ ds
= \alpha J(y,R),\quad \forall y\in (\ell,r).
\end{align*}
This, together with \eqref{J>y}, gives $J(\alpha y, R(\alpha))\ge \alpha y$ for all $y\in (\ell,r)$, which implies $(\alpha\ell,\alpha r)\subseteq I_{R(\alpha)}\cup C_{R(\alpha)}$. It follows that
\begin{equation}\label{R(alpha) in E}
R(\alpha)=(0,\alpha\ell]\cup [\alpha r,\infty)\in \cE,\quad \forall \alpha>0.
\end{equation}

Now, consider an auxiliary stopping problem with the current payoff function $f(x):=x$ replaced by $\bar f(x) := x\wedge (r+1)$, for all $x>0$. For any $x>0$ and $R\in\B(\X)$, the expected discounted payoff is then
\begin{equation*}
\bar J(x,R) := \E^x\left[\delta(\rho(x,R)) \bar f(X_{\rho(x,R)}) \right]= \E^x\left[\frac{X_{\rho(x,R)}\wedge (r+1)}{1+\beta\rho(x,R)}\right]. 
\end{equation*}
We will denote by $\bar\cE$ the collection of all equilibria for this new stopping problem. The main purpose of introducing this new problem is that since $\bar f$ is bounded, \eqref{e34} and \eqref{e32} are trivially satisfied (with $f$ replaced by $\bar f$), and thus Proposition~\ref{p2'} and Theorem~\ref{t1} can be applied to the new problem. 
For any $0<\alpha\le 1$, observe that 
\[
\bar J(y,R(\alpha)) = J(y,R(\alpha)) \ge y = \bar f(y)\quad \forall y\in (\alpha\ell, \alpha r),
\] 
where the inequality is due to \eqref{R(alpha) in E}. This implies $R(\alpha) \in \bar\cE$ for all $0<\alpha\le 1$. By taking 
$\underline\alpha := \frac{\ell +r}{2r} <1$, we may conclude from Proposition~\ref{p2'} that
\[
[r,\infty) = \bigcap_{n\in \N\cup\{0\}} R(\underline\alpha^n) \in \bar \cE. 
\]
It follows that $J(y,[r,\infty)) = \bar J(y,[r,\infty)) \ge \bar f(y) = y$ for all $y\in (0,r)$, which in turn shows that $[r,\infty)\in\cE$. 
\end{proof}

\begin{proposition}\label{prop:nu in between}
Let $\nu\in(-1/2,0]$. There are three distinct cases. 
\begin{itemize}
\item {\bf Case 1:} $\sqrt{\beta\pi/2\sigma^2}>1$. Then $(0,\infty)$ is the only closed equilibrium, for all $\nu\in(-1/2,0]$. Hence, $R^*=(0,\infty)$ is the unique closed optimal equilibrium, for all $\nu\in(-1/2,0]$.
\item {\bf Case 2:} $\sqrt{\beta\pi/2\sigma^2}=1$. 
		\begin{itemize}
		\item [(i)] if $\nu = 0$, then $[a,\infty)\in \cE$ for all $a\in(0,\infty)$. Hence, there exists no optimal equilibrium.
		\item [(ii)] if $\nu\in(-1/2,0)$, then $(0,\infty)$ is the only closed equilibrium. Hence, $R^*=(0,\infty)$ is the unique closed optimal equilibrium. 
		\end{itemize}
\item {\bf Case 3:} $\sqrt{\beta\pi/2\sigma^2}<1$. There exists $v^*\in (-1/2,0)$ such that
\begin{equation}\label{nu*}
\int_0^\infty e^{-s} \left(\sqrt{(\nu^*)^2+2\beta s/\sigma^2}-\nu^*\right) ds =1.
\end{equation}
		\begin{itemize}
		\item [(i)] if $\nu \in [v^*,0]$, then $[a,\infty)\in \cE$ for all $a\in(0,\infty)$. Hence, there exists no optimal equilibrium.
		\item [(ii)] if $\nu\in (-1/2,v^*)$, then $(0,\infty)$ is the only closed equilibrium. Hence, $R^*=(0,\infty)$ is the unique closed optimal equilibrium. 
		\end{itemize}
\end{itemize}
\end{proposition}

\begin{proof}
For each $a>0$, recall the functions $\kappa(x,a)$ in \eqref{kappa} and $g(s,\nu)$ in \eqref{g(s,nu)}. A direct calculation shows that $\nu\mapsto g(s,\nu)$ is strictly decreasing. Now, for any $\nu \le 0$, by \eqref{kappa_x} and $g(s,\nu)$ being strictly decreasing in $\nu$, 
\begin{align}\label{slope at a}
\lim_{x\uparrow a}\kappa_x(x,a) &= \int_0^\infty e^{-s} g(s,\nu)ds\notag \\
 &\ge  \int_0^\infty e^{-s} g(s,0)ds = \int_0^\infty e^{-s} \sqrt{\frac{2\beta s}{\sigma^2}}ds = \sqrt{\frac{2\beta}{\sigma^2}} \Gamma(3/2) = \sqrt{\frac{\beta\pi}{2\sigma^2}}.
\end{align}
Note that in the above relation we have strict inequality for $\nu<0$, and equality for $\nu=0$. 
\begin{itemize}[leftmargin=*]
\item {\bf Case 1:} $\sqrt{\beta\pi/2\sigma^2}>1$. For each $a>0$, by \eqref{slope at a}, $\lim_{x\uparrow a}\kappa_x(x,a)>1$ for all $\nu\le 0$. This, together with $\kappa(a,a)=a$ (Lemma~\ref{lem:kappa} (i)), implies that there exists $x^*\in (0,a)$ such that $x> \kappa(x,a)= J(x,[a,\infty))$ for all $x\in(x^*,a)$, and thus $(x^*,a)\subseteq S_{[a,\infty)}$. Hence, we conclude that $[a,\infty)\notin\cE$, for all $a>0$. This, together with Lemma~\ref{lem:two-part implies one-part}, shows that $(0,\ell]\cup [r,\infty)\notin\cE$ for all $0<\ell<r<\infty$. It follows that $(0,\infty)$ is the {\it only} closed equilibrium. We can then argue as in Proposition~\ref{prop:nu<=-1/2} to show that $R^*= (0,\infty)$ is the unique closed optimal equilibrium.  

\item {\bf Case 2:} $\sqrt{\beta\pi/2\sigma^2}=1$. For $\nu< 0$, \eqref{slope at a} gives $\lim_{x\uparrow a}\kappa_x(x,a)>1$, for any $a>0$. The same argument as in Case 1 shows that $(0,\infty)$ is the only closed equilibrium, and $R^*= (0,\infty)$ is the unique closed optimal equilibrium. 

For $\nu= 0$, \eqref{slope at a} gives $\lim_{x\uparrow a}\kappa_x(x,a)=1$, for any $a>0$. This, together with Lemma~\ref{lem:kappa}, already implies that $J(x,[a,\infty))=\kappa(x,a)>x$ for all $x\in (0,a)$. Indeed, if $\kappa(x',a) \le x'$ for some $x'\in(0,a)$, then $\lim_{x\to 0}\kappa_x(x,a)=\infty$ (Lemma~\ref{lem:kappa} (iii)), $\kappa(a,a)=a$ (Lemma~\ref{lem:kappa} (i)), and $\lim_{x\uparrow a}\kappa_x(x,a)=1$ would force $x\mapsto \kappa_{xx}(x,a)$ to change sign at least twice on $(0,a)$, which contradicts Lemma~\ref{lem:kappa} (iii). Hence, we have $(0,a)\subseteq C_{[a,\infty)}$, and thus $[a,\infty)\in\cE$, for all $a>0$. 

By contradiction, suppose there exists an optimal equilibrium $R\in\cE$. By Lemma~\ref{lem:bar R in E} and Remark~\ref{rem:value unchanged}, its closure $\overline R$ is again an optimal equilibrium. Observe that $\overline R$ cannot be either $\emptyset$ or $(0,\infty)$. Indeed, \eqref{X(infty)} yields $J(x,\emptyset) = 0<x$ for all $x>0$, which implies $\emptyset$ is not even an equilibrium. Also, for any $a>0$, $J(x,(0,\infty)) =x < J(x,[a,\infty))$ for all $x\in(0,a)$; that is, $[a,\infty)\in\cE$ generates larger values than $(0,\infty)$ on $(0,a)$, and thus $(0,\infty)$ cannot be an optimal equilibrium. 
With $\overline R\neq(0,\infty)$, we can take $x>0$ such that $x\notin\overline R$. Define $\ell$ and $r$ as in \eqref{bounds}, with $R$ replaced by $\overline R$. By the closedness of $\overline R$, we have $\ell<x$ and $r >x$. We claim that $r<\infty$ and $\ell=0$. If $r=\infty$, then $\ell>0$ must hold, otherwise $\overline R=\emptyset$. It follows that $\sup \overline R=\ell$ and $J(x,\overline R)= J(x,(0,\ell]) < \ell < x$ for all $x>\ell$, where the first inequality is deduced from \eqref{J GBM}. This, however, shows that $\overline R$ is not an equilibrium, a contradiction. Thus, we must have $r<\infty$. Now, if $\ell>0$, then $\ell \in \overline R$, and thus $V(\ell,\overline R) = \ell < J(\ell,[r,\infty))$, where the equality follows from Lemma~\ref{lem:rho R = rho bar R}. This contradicts $\overline R$ being an optimal equilibrium. Therefore, the claim that $r<\infty$ and $\ell=0$ follows, and we in particular have $\inf \overline R=r>0$. If there exists $x'>r$ such that $x'\notin \overline R$, the same argument above shows that $\inf\{y<x' : y\in \overline R\} = 0$, which contradicts $\inf \overline R=r>0$. Hence, we conclude that $\overline R = [r,\infty)$. However, for any $\eps>0$, $J(x,[r,\infty))=x < J(x,[r+\eps,\infty))$ on $(r,r+\eps)$. That is, $[r+\eps,\infty)\in\cE$ generates larger values than $\overline R=[r,\infty)$ on $(r,r+\eps)$, a contradiction to $\overline R$ being an optimal equilibrium.


\item {\bf Case 3:} $\sqrt{\beta\pi/2\sigma^2}<1$. Since $\nu\mapsto g(s,\nu)$ is strictly decreasing,
\begin{equation}\label{int > int}
\int_0^\infty e^{-s} g(s,-1/2)ds > \int_0^\infty e^{-s} g(s,0)ds. 
\end{equation}
Note that the left hand side of \eqref{int > int} is strictly larger than $1$, while the right hand side is strictly less than $1$. Indeed,
by definition, $g(s,-1/2) > 1/2+1/2=1$ for all $s>0$. It follows that $\int_0^\infty e^{-s} g(s,-1/2)ds > \int_0^\infty e^{-s} ds=1$. On the other hand, as shown in \eqref{slope at a}, $\int_0^\infty e^{-s} g(s,0)ds= \sqrt{\beta\pi/2\sigma^2}<1$. Since $\nu\mapsto g(s,\nu)$ is continuous, we conclude from \eqref{int > int} that there exists $\nu^*\in (-1/2,0)$ such that $\int_0^\infty e^{-s} g(s,\nu^*)ds =1$, i.e. \eqref{nu*} holds. 

For $\nu\in[v^*,0]$,  since $\nu\mapsto g(s,\nu)$ is strictly decreasing, $\lim_{x\uparrow a}\kappa_x(x,a) = \int_0^\infty e^{-s} g(s,\nu)ds \le 1$, for any $a>0$. This entails $J(x,[a,\infty))=\kappa(x,a)>x$ for all $x\in (0,a)$. If not, one may argue as in Case 2 that $\lim_{x\to 0}\kappa_x(x,a)=\infty$ (Lemma~\ref{lem:kappa} (iii)), $\kappa(a,a)=a$ (Lemma~\ref{lem:kappa} (i)), and $\lim_{x\uparrow a}\kappa_x(x,a)\le 1$ would force $x\mapsto\kappa_{xx}(x,a)$ to change sign at least twice on $(0,a)$, which contradicts Lemma~\ref{lem:kappa} (iii). Hence, we conclude $(0,a)\subseteq C_{[a,\infty)}$, and thus $[a,\infty)\in\cE$, for all $a>0$. This implies, as shown in Case 2, that there exists no optimal equilibrium.  

For $\nu\in(-1/2,v^*)$, since $\nu\mapsto g(s,\nu)$ is strictly decreasing, $\lim_{x\uparrow a}\kappa_x(x,a) = \int_0^\infty e^{-s} g(s,\nu)ds > 1$, for any $a>0$. One may then argue as in Case 1 to show that $(0,\infty)$ is the only closed equilibrium, and $R^*=(0,\infty)$ is the unique closed optimal equilibrium.
\end{itemize}
\end{proof}

\begin{remark}\label{rem:no optimal E}
For $\nu\in(-1/2,0]$, note from Remark~\ref{rem:conditions not satisfied} that while \eqref{e34} is satisfied, \eqref{e32} is not. Hence, Proposition~\ref{prop:nu in between} in particular shows that an optimal equilibrium may fail to exist when \eqref{e32} is violated. One can also view this from $R^*$ in \eqref{R*}.  In Case 2 (i) and Case 3 (i) in Proposition~\ref{prop:nu in between},
\[
R^* \subseteq \bigcap_{a>0} [a,\infty) = \emptyset, 
\] 
and thus $R^*=\emptyset$. However, $R^*$ is not even an equilibrium, as $J(x,\emptyset) = 0 < x$ for all $x>0$, thanks to \eqref{e34}. In other words, the validity of Theorem~\ref{t1} hinges on \eqref{e32}. 
\end{remark}

In this subsection, we show that optimal equilibria vary in forms and may fail to exist, depending on the upward potential of the geometric Brownian motion $X$. Recall that $\nu$ in \eqref{nu} measures the upward potential of $X$. For $\nu>0$, $X$ admits strongest upward potential (recall \eqref{X(infty)}). Anticipating that all future selves will choose to continue (i.e. follow the stopping policy $\emptyset$), which leads to the payoff $J(x,\emptyset) = \infty$, the current self is more than happy to continue, too. This makes $\emptyset$ (i.e. ``never stop'') an optimal equilibrium, as shown in Proposition~\ref{prop:nu>0}. For $\nu\le-1/2$, $X$ admits strongest downward potential, such that the expected payoff from any slight continuation is lower than the immediate payoff from stopping. The stopping policy $(0,\infty)$ (i.e. ``never start'') is then the only (closed) equilibrium, as established in Proposition~\ref{prop:nu<=-1/2}. 

The most intriguing case is $\nu\in(-1/2,0]$, where one needs to compare closely the mediocre upward potential of $X$ with the magnitude of discounting, measured by $\beta>0$. If discounting is severe enough relative to the volatility of $X$ (i.e. $\sqrt{\beta\pi/2\sigma^2}>1$), discounting outweighs the upward potential of $X$, such that the expected payoff from any slight continuation is lower than the immediate payoff from stopping (similar to the case $\nu\le -1/2$). The stopping policy $(0,\infty)$ (i.e. ``never start'') is then the only (closed) equilibrium, as shown in Case 1 of Proposition~\ref{prop:nu in between}. On the other hand, if discounting is not so severe relative to the volatility of $X$ (i.e. $\sqrt{\beta\pi/2\sigma^2}\le 1$), a critical level $\nu^*\in (-1/2,0]$ comes into play, as shown in Cases 2 and 3 of Proposition~\ref{prop:nu in between}. For $\nu < \nu^*$, the upward potential of $X$ is not strong enough to outweigh discounting (a similar situation as above), and $(0,\infty)$ (i.e. ``never start'') is the only (closed) equilibrium. For $\nu \ge \nu^*$, the upward potential of $X$ outweighs discounting, to the extent that continuation until any threshold $a>0$ is better than immediate stopping -- such that $[a,\infty)$ is an equilibrium for all $a>0$. Since the discount function \eqref{hyperbolic} induces decreasing impatience, the agent prefers smaller equilibria over larger ones; see Corollary~\ref{coro:smaller better} and the detailed discussion below it. This means $[a',\infty)$ is preferred over $[a,\infty)$, for all $0<a<a'$. Interestingly, despite this well-defined ``order relation'' among equilibria, there exists no optimal equilibrium. Indeed, any equilibrium $[a,\infty)$ is outperformed by another equilibrium $[a',\infty)$ with $a'>a$, and the only hopeful candidate $\bigcap_{a>0} [a,\infty) = \emptyset$ is not even an equilibrium, as pointed out in Remark~\ref{rem:no optimal E}.


\subsection{An American Put Option on a Geometric Brownian Motion} \label{subsec:put on GBM}
Consider a geometric Brownian motion $X$ given by \eqref{GBM}, with state space $\X=(0,\infty)$, as well as the payoff function $f(x):= (K-x)^+$ on $\X$, for some $K>0$, and the hyperbolic discount function specified in \eqref{hyperbolic}.  The expected payoff $J(x,R)$ in \eqref{J} then takes the form 
\begin{equation}\label{J put on GBM}
J(x,R) = \E^x\left[\frac{(K-X_{\rho(x,R)})^+}{1+\beta\rho(x,R)}\right]. 
\end{equation}
The current setting relates to pricing of a perpetual American put option, under hyperbolic discounting. Alternatively, it can be viewed as a real options problem where the management of a company considers an investment plan, which has constant payoff $K$ and stochastic cost evolving as $X$, and would like to decide when to carry it out. 

As mentioned in Section~\ref{subsec:GBM}, $X$ satisfies Assumption~\ref{a1}, thanks to Remark~\ref{rem:Assumption satisfied}. Also, since $f$ is bounded, \eqref{e34} and \eqref{e32} trivially holds. Thus,  Proposition~\ref{p2}, Theorem~\ref{t1}, and Theorem~\ref{t2} can all be applied to this stopping problem.


For $\mu\ge 0$, a clear-cut characterization of closed equilibria can be obtained. 

\begin{lemma}\label{lem:lambda>=0}
Let $\mu\ge 0$. For any $R\in\cE$ that is closed and contained in $(0,K]$, $R = (0,a]$ for some $a\in(0,K]$.
\end{lemma}

\begin{proof}
For any $R\in\cE$ that is closed and contained in $(0,K]$, set $a := \sup R\le K$. By contradiction, suppose there exists $x\in (0,a)$ such that $x\notin R$. Define $\ell$ and $r$ as in \eqref{bounds}, with $\ell$ taken to be $0$ if there exists no $y\in R$ with $y<x$.  
By the closedness of $R$, we have $\ell <x$ and $r>x$, and thus $\rho(x,R)>0$ $\P^x$-a.s. This implies
\[
J(x,R) < \E^x\left[(K-X_{\rho(x,R)})^+\right] = \E^x\left[K-X_{\rho(x,R)}\right]= K-\E^x[X_{\rho(x,R)}]. 
\]
With $\mu\ge 0$, $X$ is a submartingale and thus $\E^x[X_{\rho(x,R)}]\ge x$. The previous inequality then yields $J(x,R)< K-x= f(x)$, i.e. $x\in S_{R}$. This contradicts the fact that $R$ is an equilibrium. 
\end{proof}

To understand the behavior of $J(x,R)$, given in \eqref{J put on GBM}, we need the next technical lemma. Recall the constant $\nu\in\R$ in \eqref{nu}, and define 
\begin{equation}\label{lambda}
\lambda:=   \int_0^\infty e^{-s} \left( \sqrt{\nu^2+{2\beta s}/{\sigma^2}}+\nu \right)ds>0.
\end{equation}

\begin{lemma}\label{lem:eta}
For any $a\in(0,K)$, define
\[
\eta(x,a) := \E^x\left[\frac{K-a}{1+\beta T^x_a}\right]\quad \hbox{for}\ \ a\le x<\infty,
\]
where $T^x_a$ is defined as in \eqref{T^x_y}. 
Then, 
\begin{itemize}
\item [(i)] $x\mapsto\eta(x,a)$ is strictly decreasing and strictly convex on $(a,\infty)$, with $\eta(a,a)=K-a$ and $\lim_{x\to\infty}\eta(x,a)=0$. 
\item [(ii)] If $a< \big(\frac{\lambda}{1+\lambda}\big) K$, with $\lambda>0$ as in \eqref{lambda},
then $x\mapsto\eta(x,a)$ and $x\mapsto (K-x)^+$ intersect exactly once on $(a,\infty)$, and the intersection takes place at some $x^*\in(a,K)$. Moreover, $\eta(x,a)< (K-x)^+$ on $(a,x^*)$ and $\eta(x,a)> (K-x)^+$ on $(x^*,\infty)$.
\item [(iii)] If $a\ge \big(\frac{\lambda}{1+\lambda}\big) K$, with $\lambda>0$ as in \eqref{lambda}, then $\eta(x,a)>(K-x)^+$ on $(a,\infty)$. 
\end{itemize}
\end{lemma}

\begin{proof}
By definition, $\eta(a,a)=K-a$. Let $p(t)$ denote the density function of $T^x_a$. Observe that
\begin{align*}
\eta(x,a) &= (K-a) \int_0^\infty \frac{ p(t) }{1+\beta t}dt= (K-a) \int_0^\infty \int_0^\infty e^{-(1+\beta t)s} p(t) ds dt\\
&= (K-a) \int_0^\infty e^{-s} \E^x\left[e^{-\beta s T^x_{a}}\right] ds=(K-a) \int_0^\infty e^{-s} \left(\frac{a}{x}\right)^{\sqrt{\nu^2+\frac{2\beta s}{\sigma^2}}+\nu} ds,\quad \forall x>a.
\end{align*}
where the last equality follows from \cite[Formula 2.0.1 on p. 628]{BS-book-2002}. This particularly shows that $\lim_{x\to\infty}\eta(x,a)=0$. Moreover, for any $x>a$,
\begin{align*}
\eta_x(x,a) &= -\frac{K-a}{x} \int_0^\infty e^{-s} \left( \sqrt{\nu^2+\frac{2\beta s}{\sigma^2}}+\nu \right)\left(\frac{a}{x}\right)^{\sqrt{\nu^2+\frac{2\beta s}{\sigma^2}}+\nu} ds\ <0,\\
\eta_{xx}(x,a) &= \frac{K-a}{x^2} \int_0^\infty e^{-s} \left( \sqrt{\nu^2+\frac{2\beta s}{\sigma^2}}+\nu \right) \left( \sqrt{\nu^2+\frac{2\beta s}{\sigma^2}}+\nu+1 \right)\left(\frac{a}{x}\right)^{\sqrt{\nu^2+\frac{2\beta s}{\sigma^2}}+\nu} ds\ >0. 
\end{align*}
Thus, $x\mapsto\eta(x,a)$ is strictly decreasing and strictly convex on $(a,\infty)$. This already implies that $x\mapsto\eta(x,a)$ and $x\mapsto (K-x)^+$ intersect at most once on $(a,\infty)$, and the intersection, if exists, must happen on $(a,K)$. Observe that the intersection $x^*\in (a,K)$ exists if and only if $\lim_{x\downarrow a}\eta_x(x,a) <-1$. Since
\[
\lim_{x\downarrow a}\eta_x(x,a) = -\frac{K-a}{a} \int_0^\infty e^{-s} \left( \sqrt{\nu^2+\frac{2\beta s}{\sigma^2}}+\nu \right)ds = -\frac{K-a}{a}\lambda,
\]
$\lim_{x\downarrow a}\eta_x(x,a) <-1$ if and only if $a< \big(\frac{\lambda}{1+\lambda}\big) K$. Thus, for $a< \big(\frac{\lambda}{1+\lambda}\big) K$, the intersection $x^*\in (a,K)$ exists, and we have $\eta(x,a)< (K-x)^+$ on $(a,x^*)$ and $\eta(x,a)> (K-x)^+$ on $(x^*,\infty)$. For $a\ge \big(\frac{\lambda}{1+\lambda}\big) K$, $\eta(x,a)$ is always above $(K-x)^+$ on $(a,\infty)$.
\end{proof}

The above lemma immediately enhances the characterization in Lemma~\ref{lem:lambda>=0}. 

\begin{corollary}\label{coro:equilibrium condition}
For any $a\in (0,K]$, $R = (0,a]\in \cE$ if and only if $a\ge \big(\frac{\lambda}{1+\lambda}\big) K$, with $\lambda>0$ defined in \eqref{lambda}. 
\end{corollary}

\begin{proof}
For any $a\in (0,K]$, $R = (0,a]\in \cE$ if and only if $J(x,(0,a])= \eta(x,a)\ge (K-a)^+$ for all $x>a$. By Lemma~\ref{lem:eta}, this holds if and only if $a\ge \big(\frac{\lambda}{1+\lambda}\big) K$. 
\end{proof}

For $\mu<0$, specific forms of closed equilibria as in Lemma~\ref{lem:lambda>=0} cannot be obtained. Instead, we show that every closed equilibrium must contain a common subset. 

\begin{lemma}\label{lem:lambda<0}
Let $\mu< 0$. For any $R\in\cE$ that is closed and contained in $(0,K]$, we have 
$
\big(0,\big(\frac{\lambda}{1+\lambda}\big) K\big] \subseteq R,
$
with $\lambda>0$ defined in \eqref{lambda}.
\end{lemma}

\begin{proof}
Fix $R\in\cE$ that is closed and contained in $(0,K]$. By contradiction, suppose there exists $x\in\big(0,\big(\frac{\lambda}{1+\lambda}\big) K\big]$ such that $x\notin R$. Define $\ell$ and $r$ as in \eqref{bounds}, with $\ell$ taken to be $0$ if there exists no $y\in R$ with $y<x$. By the closedness of $R$, we have $\ell < x$ and $r> x$. Since $R$ is an equilibrium, 
\begin{equation}\label{J in (ell,r)}
J(y,R)\ge f(y)\quad  \hbox{for all $y\in (\ell,r)$}. 
\end{equation}

First, we deal with the case $r > \big(\frac{\lambda}{1+\lambda}\big) K$. Note that $R' := (0,\ell]\cup [r,\infty)$ and $\big(0,\big(\frac{\lambda}{1+\lambda}\big) K\big]$ both belong to $\cE$, thanks to  \eqref{J in (ell,r)} and Corollary~\ref{coro:equilibrium condition}, respectively. It then follows from Proposition~\ref{p2} that
$
R' \cap \big(0,\big(\frac{\lambda}{1+\lambda}\big) K\big] = (0,\ell]
$ 
again belongs to $\cE$. This, however, contradicts Corollary~\ref{coro:equilibrium condition}, as $\ell <x\le \big(\frac{\lambda}{1+\lambda}\big) K$. 

Next, we deal with the case $r \le \big(\frac{\lambda}{1+\lambda}\big) K$. 
For any $y\in(\ell,r)$, define $\tau:=\inf\{t\ge 0: X^y_t\notin (\ell,r)\}$, and let $p:(0,\infty)\times\{\ell,r\}$ denote the joint density function of $(\tau, X^y_{\tau})$. A calculation similar to \eqref{Laplace} and \eqref{J(y,R)} yields
\begin{align}\label{J(y,R)'}
J(y,R) &= (K-\ell)\int_0^\infty  e^{-s}\ h_1(s,y,\ell,r)\ ds+  (K-r)\int_0^\infty  e^{-s}\ h_2(s,y,\ell,r)\ ds, 
\end{align}
where the functions $h_1$ and $h_2$ are specified below \eqref{J(y,R)}. Observe from \eqref{J put on GBM} and \eqref{J(y,R)'} that $\int_0^\infty  e^{-s} h_1(s,y,\ell,r)\ ds < \P^y[X_\tau = \ell]$ and $\int_0^\infty  e^{-s} h_2(s,y,\ell,r)\ ds< \P^y[X_\tau = r]$. This implies
\begin{equation}\label{h_1+h_2<1}
\int_0^\infty  e^{-s} h_1(s,y,\ell,r)\ ds + \int_0^\infty  e^{-s} h_2(s,y,\ell,r)\ ds <1. 
\end{equation}
Now, pick $\alpha>1$ such that $\alpha\ell<  \big(\frac{\lambda}{1+\lambda}\big) K < \alpha r<K$. Define $R'':= (0,\alpha\ell]\cup [\alpha r,\infty)$. For any $y\in (\ell,r)$, thanks to \eqref{homogeneity}, the same calculation leading to \eqref{J(y,R)'}, through arguments in \eqref{Laplace} and \eqref{J(y,R)}, now gives
\[
J(\alpha y, R'') = (K-\alpha\ell) \int_0^\infty  e^{-s} h_1(s,y,\ell,r)\ ds + (K-\alpha r)\int_0^\infty  e^{-s} h_2(s,y,\ell,r)\ ds. 
\]
This, together with \eqref{J(y,R)'}, shows that  
\begin{align*}
J(\alpha y, R'') - \alpha J(y,R) &= -(\alpha-1) K \left[\int_0^\infty  e^{-s} h_1(s,y,\ell,r)\ ds + \int_0^\infty  e^{-s} h_2(s,y,\ell,r)\ ds \right]\\
&> -(\alpha-1) K, 
\end{align*}
where the second line follows from \eqref{h_1+h_2<1} and $\alpha>1$. As a result,
\[
J(\alpha y, R'') \ge \alpha J(y,R) -(\alpha-1) K \ge \alpha f(y) -(\alpha-1) K = K- \alpha y = f(\alpha y), 
\]
where the second inequality follows from \eqref{J in (ell,r)}. Since the above relation holds for all $y\in(\ell,r)$, we conclude that $(\alpha\ell, \alpha r)\subseteq I_{R''}\cup C_{R''}$. This implies $R''\in\cE$. Recalling that $\big(0,\big(\frac{\lambda}{1+\lambda}\big) K\big]\in \cE$ by Corollary~\ref{coro:equilibrium condition}, we conclude from Proposition~\ref{p2} that
$
R'' \cap \big(0,\big(\frac{\lambda}{1+\lambda}\big) K\big] = (0,\alpha \ell]
$ 
also belongs to $\cE$. This, however, contradicts Corollary~\ref{coro:equilibrium condition}, as $\alpha\ell<  \big(\frac{\lambda}{1+\lambda}\big) K$.
\end{proof}

Now, we are ready to present one single explicit formula for an optimal equilibrium, for any value of $\mu\in\R$.  

\begin{proposition}\label{prop:last}
$R^* = \big(0,\big(\frac{\lambda}{1+\lambda}\big) K\big]$, with $\lambda>0$ as in \eqref{lambda}, is the unique closed optimal equilibrium. 
\end{proposition}

\begin{proof}
First, observe that $(0,K]\in\cE$. Indeed, since $f\equiv 0$ on $[K,\infty)$, one trivially gets $J(x,(0,K]) = 0 = f(x)$, for all $x\in (K,\infty)$. This implies $(K,\infty)\in I_{(0,K]}$, and thus $(0,K]\in\cE$. In view of \eqref{R*} and Proposition~\ref{p2}, 
\begin{equation}\label{R* formula}
R^* = \bigcap_{R\in\cE,\ R\ \text{closed}} R\cap (0,K] = \bigcap_{R\in\cE,\ R\ \text{closed},\ R\subseteq (0,K]} R. 
\end{equation}
If $\mu\ge 0$, by \eqref{R* formula}, Lemma~\ref{lem:lambda>=0}, and Corollary~\ref{coro:equilibrium condition}, we have
\[
R^* = \bigcap_{\big(\frac{\lambda}{1+\lambda}\big) K\le a\le K} (0,a] = \bigg(0,\bigg(\frac{\lambda}{1+\lambda}\bigg) K\bigg].
\] 
If $\mu<0$, then \eqref{R* formula}, Lemma~\ref{lem:lambda<0}, and $\big(0,\big(\frac{\lambda}{1+\lambda}\big) K\big]\in\cE$ (by Corollary~\ref{coro:equilibrium condition}) again gives $R^* = \big(0,\big(\frac{\lambda}{1+\lambda}\big) K\big]$. Hence, by Theorem~\ref{t1}, $R^* = \big(0,\big(\frac{\lambda}{1+\lambda}\big) K\big]$ is an optimal equilibrium. Finally, from Lemma~\ref{lem:eta} (iii), $J(x,R^*)= \eta(x,\big(\frac{\lambda}{1+\lambda}\big) K)> (K-x)^+$ for all $x> \big(\frac{\lambda}{1+\lambda}\big) K$, which implies $(R^*)^c = C_{R^*}$. Theorem~\ref{t2} thus ensures that $R^*$ is the unique closed optimal equilibrium. 
\end{proof}

In view of the payoff function $f(x) = (K-x)^+$, one would like the geometric Brownian motion $X$ to drift below $K>0$, generating strictly positive payoff. It is then reasonable to focus on threshold-type stopping policies $(0, a]$, with $0<a\le K$. The challenge here is how small $a>0$ should be, so that the resulting policy can be an equilibrium. Intuitively, if $a>0$ is too small, the discounting involved $\frac{1}{1+\beta \rho(x,(0,a])}$ may outweigh the payoff $(K-a)^+$ obtained at stopping, reducing the expected payoff $J(x,(0,a]) = \E^x\left[\frac{(K-a)^+}{1+\beta \rho(x,(0,a])}\right]$. Thus, anticipating that future selves will follow too small a threshold $a>0$, the current self may decide to stop immediately -- such that $(0,a]$ cannot be an equilibrium. The main contribution of this subsection is finding the smallest possible threshold, i.e. $a = \big(\frac{\lambda}{1+\lambda}\big) K$, such that $(0,a]$ is an equilibrium, as shown in Lemma~\ref{lem:lambda>=0}, Corollary~\ref{coro:equilibrium condition}, and Lemma~\ref{lem:lambda<0}. Now, since the discount function \eqref{hyperbolic} induces decreasing impatience, the agent prefers smaller equilibria over larger ones; see Corollary~\ref{coro:smaller better} and the detailed discussion below it. This readily shows, at least at the intuitive level, that the smallest possible equilibrium $\big(0,\big(\frac{\lambda}{1+\lambda}\big) K\big]$ is optimal, which is rigorously established in Proposition~\ref{prop:last}.

The threshold $\big(\frac{\lambda}{1+\lambda}\big) K$ responds desirably to the downward potential of $X$. Similarly to Subsection~\ref{subsec:GBM}, $\nu$ in \eqref{nu} measures the downward potential of $X$: the smaller $\nu$, the stronger the downward potential. 
It can be checked that $\lambda>0$ in \eqref{lambda} is strictly increasing in $\nu$, so the downward potential of $X$ can also be measured by how small $\lambda$ is. Observe that $\frac{\lambda}{1+\lambda}$ is strictly increasing in $\lambda$, with $\lim_{\lambda\downarrow 0} \frac{\lambda}{1+\lambda} =0$ and $\lim_{\lambda\uparrow \infty} \frac{\lambda}{1+\lambda} =1$. When $\lambda$ is very small (i.e. $X$ has strong downward potential), one should take a stopping threshold close to $0$, to fully exploit the downward potential that can generate strictly positive payoff. This is captured by $\big(\frac{\lambda}{1+\lambda}\big) K$, decreasing to 0 as $\lambda\downarrow 0$. As $\lambda$ increases (i.e. the downward potential of $X$ weakens), $\big(\frac{\lambda}{1+\lambda}\big) K$ also increases, reflecting an inclination to stop earlier.  
When $\lambda$ is very large (i.e. $X$ has little downward potential), one should stop once $X$ drifts only slightly below $K>0$. This is again captured by $\big(\frac{\lambda}{1+\lambda}\big) K$, increasing to $K$ as $\lambda\uparrow \infty$.


\end{document}